\RequirePackage{etex}
\documentclass[12pt,a4paper]{amsart}

\usepackage[T1]{fontenc}
\usepackage{lmodern}
\usepackage[english]{babel}
\usepackage{textcomp}

\usepackage{amsmath, amsthm, amssymb, mathtools}
\usepackage{mathrsfs}
\usepackage{mathrsfs}
\usepackage[cal=rsfso,scaled=1.05]{mathalfa}
\usepackage{bbm, bm, stmaryrd, upgreek}
\usepackage{fixmath}

\usepackage{tikz-cd}
\usepackage[all]{xy}

\usepackage{graphicx}
\graphicspath{ {./images/} }
\usepackage{pdfpages}
\usepackage{float}
\usepackage{array, booktabs, tabularx, multirow, multicol}
\usepackage{xurl}
\usepackage{ragged2e}

\makeatletter
\renewcommand\normalsize{%
    \@setfontsize\normalsize{11.7}{14pt plus .3pt minus .3pt}%
    \abovedisplayskip 10\p@ \@plus4\p@ \@minus4\p@
    \abovedisplayshortskip 6\p@ \@plus2\p@
    \belowdisplayshortskip 6\p@ \@plus2\p@
    \belowdisplayskip \abovedisplayskip}
\renewcommand\small{%
    \@setfontsize\small{9.5}{12\p@ plus .2\p@ minus .2\p@}%
    \abovedisplayskip 8.5\p@ \@plus4\p@ \@minus1\p@
    \belowdisplayskip \abovedisplayskip
    \abovedisplayshortskip \abovedisplayskip
    \belowdisplayshortskip \abovedisplayskip}
\renewcommand\footnotesize{%
    \@setfontsize\footnotesize{8.5}{9.25\p@ plus .1pt minus .1pt}%
    \abovedisplayskip 6\p@ \@plus4\p@ \@minus1\p@
    \belowdisplayskip \abovedisplayskip
    \abovedisplayshortskip \abovedisplayskip
    \belowdisplayshortskip \abovedisplayskip}

\ifdefined\pdfpagewidth
\else
\fi
\calclayout
\makeatother

\usepackage{enumitem}
\setlist{nosep}
\setlist[enumerate]{label=(\roman*)}

\newlist{defenum}{enumerate}{3}
\setlist[defenum]{label=(\alph*),leftmargin=*,align=left}
\newlist{defsubenum}{enumerate}{1}
\setlist[defsubenum]{label=(\roman*),leftmargin=*,align=left}

\usepackage{etoolbox}
\AtBeginEnvironment{thebibliography}{%
  \setlength{\itemsep}{0pt}%
  \setlength{\parskip}{0pt}%
  \setlength{\topsep}{0pt}%
  \setlength{\partopsep}{0pt}%
}

\theoremstyle{plain}
\newtheorem{theorem}{Theorem}[section]
\newtheorem{lemma}[theorem]{Lemma}
\newtheorem{proposition}[theorem]{Proposition}
\newtheorem{corollary}[theorem]{Corollary}

\theoremstyle{definition}
\newtheorem{definition}[theorem]{Definition}
\newtheorem{example}[theorem]{Example}

\theoremstyle{remark}

\DeclareMathOperator{\id}{id}

\DeclareMathOperator{\Ad}{Ad}

\DeclareMathOperator{\Ricci}{Ric}

\newcommand{\ga}{\mathit{g}} 

\newcolumntype{L}{>{$}l<{$}}
\newcolumntype{C}{>{$}c<{$}}
\newcolumntype{R}{>{$}r<{$}}

\usepackage[colorlinks,pagebackref,hypertexnames=false]{hyperref}
\hypersetup{linkcolor=red, citecolor=blue, filecolor=blue, urlcolor=blue}

\usepackage[nameinlink,noabbrev]{cleveref}
\crefname{theorem}{Theorem}{Theorems}
\crefname{Mtheorem}{Main Theorem}{Main Theorems}
\crefname{lemma}{Lemma}{Lemmata}
\crefname{corollary}{Corollary}{CorCorollaries}
\crefname{proposition}{Proposition}{Propositions}
\crefname{definition}{Definition}{Definitions}

\usepackage{autonum}

\usepackage[alphabetic,backrefs]{amsrefs}

\title{The $G$-equivariant Kazdan--Warner problem}

\author[Cavenaghi]{Leonardo F. Cavenaghi}
\address{Institute of Mathematics and Informatics, Bulgarian Academy of Sciences}
\email{leonardofcavenaghi@gmail.com}

\author[do \'O]{João Marcos do Ó}
\address{Departamento de Matemática -- Universidade Federal da Paraíba - Campus 1, 1st floor - Lot. Cidade Universitária
58051-900, João Pessoa, PB, Brazil}
\email{jmbo@mat.ufpb.br}

\author[Speran\c ca]{Llohann D. Speran\c ca}
\address{Instituto de Ciência e Tecnologia -- Unifesp, Avenida Cesare Mansueto Giulio Lattes, 1201, 12247-014, São José dos Campos, SP, Brazil}
\email{lsperanca@gmail.com}

\begin{document}

\begin{abstract}
We establish an equivariant analogue of the Kazdan--Warner trichotomy for admissible scalar curvature functions. Let $M$ be a closed connected manifold of dimension $n \ge 3$ equipped with an effective isometric action of a compact connected Lie group $G$ of cohomogeneity at least one and with no zero-dimensional orbits. All metrics and prescribed functions are taken to be $G$-invariant. We prove that, for such pairs $(M, G)$, the classical trichotomy does not extend verbatim. A distinct class emerges, consisting of \emph{totally $G$-positive} pairs, for which every $G$-invariant metric exhibits positive total scalar curvature. Each such pair admits a $G$-invariant metric of positive constant scalar curvature, but admits no metric of zero or negative constant scalar curvature. Every remaining pair falls into exactly one of three classes mirroring the original trichotomy.
\end{abstract}

\maketitle

\section{Introduction}\label{sec:intro}

A landmark result of Kazdan and Warner \cite{kazadanannals} determines, for a connected closed manifold $M^n$ with $n \ge 3$, which smooth functions arise as the scalar curvature of some Riemannian metric. Their classification sorts all such manifolds into a trichotomy: every connected closed $M^n$ ($n \ge 3$) belongs to exactly one of three mutually exclusive classes.
\begin{itemize}
  \item[($\mathscr{P}$)] $M$ admits a metric of positive scalar curvature; in this case, \emph{every} $f \in C^\infty(M)$ is the scalar curvature of some metric.
  \item[($\mathscr{Z}$)] $M$ admits a scalar-flat metric but no metric of positive scalar curvature; in this case, $f \in C^\infty(M)$ is a scalar curvature if and only if $f \equiv 0$ or $f$ is negative somewhere.
  \item[($\mathscr{N}$)] $M$ admits no metric of non-negative scalar curvature; in this case, $f \in C^\infty(M)$ is a scalar curvature if and only if $f$ is negative somewhere.
\end{itemize}

This paper investigates the equivariant analogue of this classification. Throughout, $M$ is a connected closed smooth manifold of dimension $n \ge 3$, and $G$ is a compact connected Lie group acting effectively and isometrically on $M$. We assume the action has \emph{cohomogeneity at least one} (it is not transitive) and features \emph{no zero-dimensional orbits}. Restricting to $G$-invariant metrics and $G$-invariant prescribed functions, one asks: which $G$-invariant functions are the scalar curvature of a $G$-invariant metric, and do the pairs $(M,G)$ partition into three classes mirroring $\mathscr{P}$, $\mathscr{Z}$, and $\mathscr{N}$?

Our main result (Theorem~\ref{ithm:main}) proves that the Kazdan--Warner trichotomy does not extend verbatim. A new alternative emerges, possessing no counterpart in the non-equivariant setting. 

\begin{definition}\label{def:totally}
A pair $(M,G)$ is \emph{totally $G$-positive} if $\int_M \mathrm{scal}_{\ga}\,\mathrm{d}\mu_{\ga} > 0$ for every smooth $G$-invariant Riemannian metric $\ga$ on $M$.
\end{definition}

A totally $G$-positive pair cannot admit a $G$-invariant metric whose scalar curvature is everywhere non-positive. In particular, it admits no $G$-invariant metric of constant zero or negative scalar curvature. Nevertheless, it always admits a metric of positive constant scalar curvature. Such a pair therefore behaves like the class $\mathscr{P}$ in its capacity to carry positive scalar curvature, while failing the defining property of the equivariant class $\mathscr{P}^G$, namely the ability to prescribe any smooth $G$-invariant function as a scalar curvature.

Under the standing hypotheses of this paper, total $G$-positivity admits the following geometric characterization (Lemma~\ref{lem:dichotomy}): \emph{$(M, G)$ is totally $G$-positive if and only if the action is of cohomogeneity one with an isotropy-irreducible principal orbit (Definition \ref{def:isotropy-irreducible})}. For a totally $G$-positive pair $(M,G)$, a principal orbit carries a unique $G$-invariant metric $\ga_P$ up to scaling, ensuring that \emph{every} $G$-invariant metric on $M$ is a warped product $\ga = \mathrm{d}r^2 + \phi(r)^2 \ga_P$ (Lemma \ref{lem:smooth}). The orbit space is homeomorphic to one of two spaces. Either it is a circle, with $\phi$ being periodic (Example~\ref{eq:product-of} serves as the prototype), or it is a closed interval whose endpoints are \emph{exceptional} orbits, at which $\phi$ remains strictly positive but $\phi'$ vanishes (Example~\ref{ex:Bohm}). 

\begin{example}\label{eq:product-of}
Consider $M = S^1 \times S^{n-1}$ for $n \ge 3$, equipped with the action of $G = SO(n)$ defined by $g \cdot (z, x) = (z, gx)$. This action is effective and of cohomogeneity one, with all orbits being principal and of the form $\{z\} \times S^{n-1} \cong SO(n)/SO(n-1)$, possessing isotropy $SO(n-1)$. The isotropy representation is the standard action of $SO(n-1)$ on $\mathbb{R}^{n-1}$, which is irreducible over $\mathbb{R}$ for $n \ge 3$. Consequently, it has no non-zero fixed vectors, ensuring the orbits carry no non-trivial $G$-invariant $1$-forms.

Let $\ga$ be an arbitrary $G$-invariant Riemannian metric on $M$. Let $TM = \mathcal{H} \oplus \mathcal{V}$ denote the smooth direct sum of the orbit tangent spaces $\mathcal{V}$ and the line bundle $\mathcal{H} = \operatorname{span}(\partial_r)$. Because $G$ acts trivially on the $S^1$-factor, the vector field $\partial_r$ is $G$-invariant. Thus, the restriction $\omega := \ga(\partial_r,\cdot)|_{\mathcal{V}}$ defines a $G$-invariant $1$-form on each orbit. Since the orbits admit no non-zero $G$-invariant $1$-forms, $\omega \equiv 0$. This implies $\mathcal{H} \perp_{\ga} \mathcal{V}$, dictating that the metric splits orthogonally as $\ga = \ga_{\mathcal{H}} \oplus \ga_{\mathcal{V}}$.

After an appropriate arc-length reparametrization, the horizontal component becomes $\mathrm{d}r^2$ on $S^1$. For the vertical component, the irreducibility of the isotropy representation guarantees that the orbit $S^{n-1}$ carries a unique $G$-invariant metric up to scaling. Hence, $\ga$ is necessarily a warped product
\begin{equation}
    \ga = \mathrm{d}r^2 + \phi(r)^2 g_{\mathrm{round}}, \quad \text{where } \phi \colon S^1 \to (0,\infty) \text{ is smooth.}
\end{equation}

By \cite{besse1987einstein}*{Proposition~9.106}, the scalar curvature of the warped product $S^1 \times_\phi S^{n-1}$ is given by
\begin{equation}
    \mathrm{scal}_{\ga} = \frac{(n-1)(n-2)}{\phi^2} - 2(n-1)\frac{\phi''}{\phi} - (n-1)(n-2)\frac{(\phi')^2}{\phi^2}.
\end{equation}
The volume form is $\mathrm{d}\mu_{\ga} = \phi^{n-1}\,\mathrm{d}r \wedge \mathrm{d}\mu_{\mathrm{round}}$. Integration by parts yields
\begin{align}
    \int_M \mathrm{scal}_{\ga} \,\mathrm{d}\mu_{\ga} &= \mathrm{Vol}(S^{n-1})\,(n-1)(n-2)\int_{S^1} \Bigl[ \phi^{n-3} + (\phi')^2\,\phi^{n-3} \Bigr] \mathrm{d}r > 0.
\end{align}
\end{example}

\ 

\begin{theorem}\label{ithm:main}
The following classification holds for pairs $(M,G)$.
\begin{enumerate}
\item If $(M,G)$ is totally $G$-positive, then it admits a $G$-invariant metric of positive constant scalar curvature, while the scalar curvature of any $G$-invariant metric must be positive somewhere. In particular, it admits no $G$-invariant metric of zero or negative constant scalar curvature, and it belongs to none of the classes below.
\item Otherwise, $(M,G)$ belongs to exactly one of the following classes:
\begin{itemize}
\item[\textup{(}$\mathscr{P}^G$\textup{)}] every smooth $G$-invariant function is the scalar curvature of a $G$-invariant metric on $M$;
\item[\textup{(}$\mathscr{Z}^G$\textup{)}] a smooth $G$-invariant function is the scalar curvature of a $G$-invariant metric if and only if it is identically zero or negative somewhere;
\item[\textup{(}$\mathscr{N}^G$\textup{)}] a smooth $G$-invariant function is the scalar curvature of a $G$-invariant metric if and only if it is negative somewhere.
\end{itemize}
Furthermore, if $(M, G)$ is not totally $G$-positive and the Lie algebra of $G$ is non-abelian, then $(M, G) \in \mathscr{P}^G$.
\end{enumerate}
\end{theorem}

Theorem~\ref{ithm:main} diverges from the classical Kazdan--Warner trichotomy in two fundamental ways. First, the trichotomy is no longer exhaustive: the totally $G$-positive pairs constitute a genuine fourth alternative lying outside $\mathscr{P}^G$, $\mathscr{Z}^G$, and $\mathscr{N}^G$. By Lemma~\ref{lem:dichotomy}, these are precisely the cohomogeneity-one actions with an isotropy-irreducible principal orbit, indicating that the intrinsic geometry of the action governs this phenomenon. We emphasize that while $S^1 \times S^{n-1} \in \mathscr{P}$, the equivariant pair $(S^1 \times S^{n-1}, SO(n)) \notin \mathscr{P}^{SO(n)}$. Second, once the family of totally $G$-positive pairs is excluded, the classical trichotomy is recovered verbatim at the level of $G$-invariant data, preserving the identical criteria for admissible scalar curvatures. The non-abelian stipulation relies on a classical result by Lawson and Yau \cite{lawson-yau}: a non-abelian symmetry group guarantees the existence of a $G$-invariant metric of positive scalar curvature. We stress that this implication is confined to case~(2): the action detailed in Example~\ref{eq:product-of} is non-abelian yet totally $G$-positive, placing it in case~(1) rather than the class $\mathscr{P}^G$. 

\ 

Theorem~\ref{ithm:main} is deduced from a $G$-invariant Yamabe-type result, heavily influenced by the work of Hebey and Vaugon \cites{hebey1, hebey2, hebey3}.

\begin{theorem}\label{theorem:invariantyamabe}
For any pair $(M,G)$:
\begin{enumerate}[label=\textup{(\arabic*)}, font=\upshape]
  \item If $(M,G)$ is not totally $G$-positive, it admits a $G$-invariant metric of constant negative scalar curvature. Furthermore, if it admits a $G$-invariant metric of non-negative scalar curvature, it admits one of identically zero scalar curvature.
  \item If $(M,G)$ admits a $G$-invariant metric of non-identically-zero non-negative scalar curvature, it admits a $G$-invariant metric of positive constant scalar curvature.
  \item If $(M, G)$ is of cohomogeneity one with an isotropy-irreducible principal orbit, then it is totally $G$-positive. In this case, it admits a $G$-invariant metric of positive constant scalar curvature, but no metric whose scalar curvature is everywhere non-positive.
\end{enumerate}
\end{theorem}

\ 

We conclude the Introduction by formally stating the auxiliary results required for the classification.

\ 

Let $X = M/G$ denote the orbit space and $\pi \colon M \to X$ the quotient projection. By \cite{alexandrino2015lie}*{\S 3.4}, the set of principal orbits $M^{\mathrm{princ}}$ is an open dense $G$-invariant subset of $M$, and $X^* = M^{\mathrm{princ}}/G$ is a connected smooth manifold with $\dim X^* = \dim X$. Its complement, $X \setminus X^*$, consists of non-principal orbits, encompassing both singular and exceptional types. In the cohomogeneity-one setting, these non-principal orbits correspond exactly to the two endpoints when $X \cong [0, L]$, and they are entirely absent when $X \cong S^1$. Every $G$-invariant function $f \colon M \to \mathbb{R}$ descends to a unique continuous function $\bar{f} \colon X \to \mathbb{R}$ satisfying $f = \bar{f} \circ \pi$.

\begin{theorem}\label{theorem:GKW}
Let $\ga$ be a $G$-invariant Riemannian metric on $M$, and let $f \in C^\infty_G(M)$ be a smooth $G$-invariant function. If there exists a constant $c > 0$ satisfying one of the following conditions, then $M$ admits a $G$-invariant Riemannian metric with scalar curvature $f$:
\begin{enumerate}[label=\textup{(\alph*)}, font=\upshape]

  \item $\dim X \ge 2$, and the strict inequalities
  $$c\min_M f < \mathrm{scal}_{\ga}(x) < c\max_M f$$
  hold for all $x \in M$.

  \item $X$ is homeomorphic to a compact interval $[0, L]$ whose boundary points correspond to the singular or exceptional orbits of the action. Furthermore, $f$ is non-constant, and there exists a continuous weakly monotone map $h \colon X \to X$ fixing the endpoints $0$ and $L$, such that the induced functions on $X$ satisfy
  $$\overline{\mathrm{scal}}_{\ga} = c\,\bar{f} \circ h.$$

  \item $X$ is diffeomorphic to $S^1$, meaning all orbits of the action are principal. Furthermore, $f$ is non-constant, and there exists a continuous weakly monotone map $h \colon S^1 \to S^1$ of degree one such that the induced functions on $X \cong S^1$ satisfy
  $$\overline{\mathrm{scal}}_{\ga} = c\,\bar{f} \circ h.$$
  
\end{enumerate}
\end{theorem}

A foundational technical step for this prescription is the following equivariant approximation result, which allows for the rearrangement of $G$-invariant functions via $G$-equivariant diffeomorphisms. This theorem extends the Kazdan--Warner Approximation Lemma \cite{kazadanannals}*{Theorem~2.1} to isometric $G$-actions of cohomogeneity at least one.

\begin{theorem}\label{lem:approximationintro}
Let $\ga$ be a $G$-invariant Riemannian metric on $M$, and assume that the $G$-action is isometric. Let $f_1 \in C^0_G(M)$ and $f_2 \in L^p_G(M)$ for $1 \le p < \infty$. For any $\epsilon > 0$, there exists a $G$-equivariant diffeomorphism $\widetilde{\Phi} \in \mathrm{Diff}_G(M)$ such that
$$\| f_1 \circ \widetilde{\Phi} - f_2 \|_{L^p(M)} < \epsilon,$$
provided one of the following conditions holds:
\begin{enumerate}[label=\textup{(\alph*)}, font=\upshape]
  \item $\dim X \ge 2$, and $\min_M f_1 \le f_2(x) \le \max_M f_1$ for almost every $x \in M$.
  \item $X$ is homeomorphic to a compact interval $[0, L]$ whose boundary points correspond to the singular or exceptional orbits of the action. The conclusion holds if either of the following sub-conditions is satisfied:
  \begin{enumerate}[label=\textup{(b.\arabic*)}, font=\upshape]
    \item $f_2 \equiv c$ is constant with $c \in [\min_M f_1,\, \max_M f_1]$; or
    \item there exists a continuous weakly monotone map $h \colon X \to X$ fixing the boundary points $0$ and $L$, such that $\bar{f}_2 = \bar{f}_1 \circ h$ on $X$.
  \end{enumerate}
  \item $X$ is diffeomorphic to $S^1$, meaning all orbits of the $G$-action on $M$ are principal. The conclusion holds if either of the following sub-conditions is satisfied:
  \begin{enumerate}[label=\textup{(c.\arabic*)}, font=\upshape]
    \item $f_2 \equiv c$ is constant with $c \in [\min_M f_1,\, \max_M f_1]$; or
    \item there exists a continuous weakly monotone map $h \colon S^1 \to S^1$ of degree one such that $\bar{f}_2 = \bar{f}_1 \circ h$ on $X \cong S^1$.
  \end{enumerate}
\end{enumerate}
\end{theorem}

\ 

\subsection*{Organization of the Paper}
This paper is organized as follows. In Section~\ref{sec:direct}, we develop a direct approach to prescribing $G$-invariant scalar curvature, analyzing the linearization of the scalar curvature operator and establishing local surjectivity criteria. Section~\ref{sec:Yamabe} is devoted to the equivariant Yamabe problem, where we construct $G$-invariant metrics of constant scalar curvature using variational methods and investigate the topological constraints imposed by total $G$-positivity. Finally, Section~\ref{sec:classification} synthesizes these results to prove the primary classification given in Theorem~\ref{ithm:main}.

\

\

\section{A direct approach to prescribing \texorpdfstring{$G$}{G}-invariant scalar curvature}
\label{sec:direct}

We denote by $W_G^{2,p}(M; \mathrm{S}^2T^{\ast}M)$, for $p>n$, the Sobolev space of $G$-invariant symmetric $(0,2)$-tensors on $M$. The Sobolev Embedding Theorem implies that any tensor in this space is continuously differentiable. 

The scalar curvature operator $F\colon\ga\mapsto\mathrm{scal}_{\ga}$ is defined on the open cone $\mathcal{M}_G^{2,p} \subset W_G^{2,p}(M; \mathrm{S}^2T^{\ast}M)$ (which is non-empty because $G$ is compact) of $G$-invariant positive-definite symmetric tensors (Riemannian metrics). For any $\ga \in \mathcal{M}_G^{2,p}$, the problem of prescribing a smooth $G$-invariant function $K \in C^\infty_G(M)$ as the scalar curvature of a Riemannian metric on $M$ is equivalent to solving for $\ga$ the quasilinear partial differential equation:
\begin{equation}\label{eq:scalar_eqn}
    F(\ga) = K.
\end{equation}
Since the $G$-action on $\ga$ is via isometries, the scalar curvature $F(\ga)$ is $G$-invariant, and $F$ is well-defined as a map $F: \mathcal{M}_G^{2,p} \to L^p_G(M)$, where $L^p_G(M)$ denotes the Lebesgue space of $G$-invariant functions on $M$.

The linearization of $F$ at $\ga$ in the direction of a variation $h \in W_G^{2,p}(M; \mathrm{S}^2T^{\ast}M)$ is the linear operator:
\begin{equation}
    Ah := \left. \frac{d}{dt} \right|_{t=0} F(\ga + th).
\end{equation}
Following \cite{FISHERMarsden}*{Proof of Theorem 3}, $A$ is expressed as:
\begin{equation}
    Ah = -\Delta_{\ga} (\mathrm{tr}_{\ga} h) + \delta_{\ga} \delta_{\ga} h - \langle h, \mathrm{Ric}(\ga)\rangle_{\ga},
\end{equation}
where $\delta_{\ga}$ and $\Delta_{\ga}$ denote the divergence and the Laplace-Beltrami operator, respectively. Here, we adopt the conventions $\Delta_{\ga} u = \mathrm{tr}(\nabla^2 u)$, where $\nabla^2 u$ is the Hessian of $u$, and $\delta_{\ga} h = -\nabla^j h_{ij}$. The formal adjoint $A^*$ with respect to the $G$-invariant $L^2$ inner product induced by $\ga$ is given by (cf. \cite{Lichnerowicz1961}*{Section 8} or \cite{FISHERMarsden}*{Proof of Theorem 3}):
\begin{equation}\label{eq:adjoint}
    A^*u = -(\Delta_{\ga} u)\ga + \nabla^2 u - u \mathrm{Ric}(\ga).
\end{equation}

\begin{lemma}\label{lem:marsden}
Let $\ga \in \mathcal{M}_G^{2,p}$ be smooth, and let $A^*$ be the formal adjoint \eqref{eq:adjoint} of the linearization of $F$ at $\ga$. Throughout, $\ker A^*$ refers to the $G$-invariant solutions of $A^* u = 0$.
    \begin{enumerate}[label=\textup{(\alph*)}, font=\upshape]
        \item If $\ker A^*$ is non-trivial, then $F(\ga)$ is a positive constant or $\mathrm{Ric}(\ga) = 0$. 
        \item If $\ker A^* = \{0\}$, then $L = AA^* : W_G^{4,p}(M) \to L_G^p(M)$ is an isomorphism. In particular, $A: W_G^{2,p}(M; \mathrm{S}^2T^{\ast}M) \to L_G^p(M)$ is surjective.
    \end{enumerate}
\end{lemma}
\begin{proof}
    \begin{enumerate}[label=\textup{(\alph*)}, font=\upshape]
        \item Let $u \in L^2_G(M)$, $u \not\equiv 0$, satisfy $A^*u = 0$ (a priori in the distributional sense). Taking the trace of Equation \eqref{eq:adjoint} yields
        $$(n-1)\Delta_{\ga} u + F(\ga)u = 0. \quad \text{(Trace Eq)}$$
        This is a scalar second-order elliptic equation for $u$ with smooth coefficients (as $\ga$ is smooth), so elliptic regularity gives $u \in C^\infty_G(M)$, and all derivatives appearing below are classical.
        
        Assume that $F(\ga)$ is not a positive constant. We will show that $\mathrm{Ric}(\ga) = 0$. We split into two cases for $u$:
        \begin{itemize}
        \item[(a.i)] Suppose $u$ never vanishes. Applying the divergence operator $\delta_{\ga}$ to the equation $A^*u = 0$ provides a critical identity. Under the convention $\delta_{\ga} = -\nabla^j(\,\cdot\,)_{ij}$, we evaluate the divergence of each term in $A^*u = -(\Delta_{\ga} u)\ga + \nabla^2 u - u\,\mathrm{Ric}(\ga)$. Specifically, $\delta_{\ga}\bigl(-(\Delta_{\ga}u)\ga\bigr) = \mathrm{d}(\Delta_{\ga}u)$. By the Ricci identity, $\delta_{\ga}(\nabla^2u) = -\mathrm{d}(\Delta_{\ga}u) - \mathrm{Ric}(\ga)(\nabla u,\cdot)$. Finally, $\delta_{\ga}\bigl(-u\,\mathrm{Ric}(\ga)\bigr) = -u\,\delta_{\ga}\mathrm{Ric}(\ga) + \mathrm{Ric}(\ga)(\nabla u,\cdot)$. Summing these contributions and applying the contracted Bianchi identity $\delta_{\ga}\mathrm{Ric}(\ga) = -\tfrac{1}{2}\mathrm{d} F(\ga)$, all Hessian and $\mathrm{Ric}(\ga)(\nabla u,\cdot)$ terms strictly cancel, leaving $\tfrac{1}{2}\,u\,\mathrm{d} F(\ga) = 0$. Because $u \neq 0$ everywhere by assumption, we deduce that $\mathrm{d} F(\ga) = 0$. Thus, $F(\ga)$ is constant.
          \item[(a.ii)] Suppose $u$ vanishes at a point $x \in M$. If $\mathrm{d} u(x) = 0$, we can substitute the trace identity $\Delta_{\ga} u = -\frac{F(\ga)}{n-1}u$ into the equation $A^*u = 0$, yielding:
          $$ \nabla^2 u = u \left( \mathrm{Ric}(\ga) - \frac{F(\ga)}{n-1}\ga \right). $$
          For any unit-speed geodesic $\gamma$ starting at $x$, the restriction $h(s) = u(\gamma(s))$ satisfies the linear second-order ODE:
          $$ h''(s) = h(s) \left( \mathrm{Ric}(\dot{\gamma}, \dot{\gamma}) - \frac{F(\ga)}{n-1} \right). $$
          With vanishing initial conditions $h(0)=0$ and $h'(0)=0$, the uniqueness of solutions for this linear second-order ODE implies $h(s) \equiv 0$, so $u$ vanishes along every geodesic emanating from $x$ (the geodesics from $x$ cover a neighbourhood of $x$ via $\exp_x$, so $u$ vanishes on an open set. The vanishing set is then open and closed). Since $M$ is complete and connected, we conclude that $u \equiv 0$, which contradicts the assumption that $u \not\equiv 0$. Thus, $0$ is a regular value of $u$. The set $M \setminus u^{-1}(0)$ is open and dense. On this set, the argument from case (a.i) guarantees $\mathrm{d} F(\ga) = 0$. By continuity, $F(\ga)$ is constant on $M$.
        \end{itemize}
         In both cases, $F(\ga)$ must be a constant. Multiplying the (Trace Eq) by $u$ and integrating over $M$ yields:$$\int_M |\mathrm{d} u|^2 \,\mathrm{d}\mu_{\ga} = \frac{F(\ga)}{n-1} \int_M u^2 \,\mathrm{d}\mu_{\ga}\implies F(\ga) \geq 0.$$ Since $F(\ga)$ is constant and non-negative, but assumed not to be a positive constant, we must have $F(\ga) = 0$.
         
         Substituting $F(\ga) = 0$ back into the (Trace Eq) implies $\Delta_{\ga} u = 0$. Since $M$ is compact, the harmonic function $u$ must be a constant $c \neq 0$. Finally, substituting $u = c$ and $\nabla^2 u = 0$ into $A^*u = 0$ leaves $-c\,\mathrm{Ric}(\ga) = 0$, which implies $\mathrm{Ric}(\ga) = 0$.
        \item Assume $\ker A^* = \{0\}$. Consider the fourth-order linear operator $L := AA^* : W_G^{4,p}(M) \to L_G^p(M)$. The principal symbol of $L$ is $\sigma_L(\xi) = (n-1)|\xi|^4$, so $L$ is elliptic. Since $L$ commutes with the $G$-action and $G$ is compact, its restriction to the invariant component $L \colon W_G^{4,p}(M) \to L_G^p(M)$ is again elliptic and formally self-adjoint, hence Fredholm. Because $L$ is elliptic on a closed manifold, its Fredholm index is independent of $p$. Since $L$ is formally self-adjoint, its index on $L^2_G$ is zero, so its index on $W_G^{4,p}$ is also zero.
    
     The identity $\int_M u (AA^* u) \mathrm{d}\mu_{\ga} = \int_M |A^* u|^2 \mathrm{d}\mu_{\ga}$ implies that $\ker(L) = \ker(A^*)$. By hypothesis, $\ker A^* = \{0\}$, hence $L$ is injective. Because its index is zero, $L$ is an isomorphism.
    
    Given any $f \in L_G^p(M)$, let $u = L^{-1} f \in W_G^{4,p}(M)$. Since $f$ is $G$-invariant and $L$ commutes with the $G$-action, $u$ is also $G$-invariant. Set $h := A^* u$. Then $h \in W_G^{2,p}(M; \mathrm{S}^2T^{\ast}M)$ and $Ah = A(A^* u) = L u = f.$
    \end{enumerate}
\end{proof}

\ 

\begin{lemma}\label{lem:inversefunction}
Let $\ga \in \mathcal{M}_G^{2,p}$ be a smooth metric. Let $A = F'(\ga)$ and assume that $\ker A^* = \{0\}$. Then $F : \mathcal{M}_G^{2,p} \to L^p_G(M)$ is locally surjective at $\ga$. Moreover, if the prescribed scalar curvature $K$ is smooth, any solution $\ga_{\mathrm{new}}$ obtained via this method is also smooth.
\end{lemma}

\begin{proof}
Consider the map $Q: U \subset W^{4,p}_G(M) \to L^p_G(M)$ defined by $Q(u) = F(\ga + A^*u)$, where $U$ is a sufficiently small neighborhood of $0$ such that $\ga + A^*u$ remains positive-definite. Notice that such $U$ exists because $\mathcal{M}_G^{2,p}$ is open and $A^*\colon W^{4,p}_G(M) \to W^{2,p}_G(M) \hookrightarrow C^0$ is bounded. The map $Q$ is a smooth mapping of Banach spaces with its Fr\'echet derivative at the origin given by 
\begin{equation}
    Q'(0) = F'(\ga) \circ A^* = AA^*.
\end{equation}

Lemma \ref{lem:marsden} says that if $\ker A^* = \{0\}$, then $AA^*: W^{4,p}_G(M) \to L^p_G(M)$ is an isomorphism. By the Inverse Function Theorem for Banach spaces, $Q$ is a local diffeomorphism, mapping $U$ onto a neighborhood of $F(\ga)$ in $L^p_G(M)$. Thus, for any $K$ sufficiently close to $F(\ga)$ in $L^p_G(M)$, there exists $u \in W^{4,p}_G(M)$ such that $F(\ga + A^*u) = K$. 

Finally, $Q(u)=K$ is a fourth-order elliptic equation whose linearization at $u=0$ is the
elliptic operator $L=AA^*$ with smooth coefficients (as $\ga$ is smooth). Interior elliptic
regularity for nonlinear elliptic equations \cite{Morrey1966}*{Theorem~6.8.1} (bootstrapping
$u\in W^{4,p}\Rightarrow u\in C^\infty$ when $K\in C^\infty$) gives $u\in C^\infty_G(M)$.
Hence $\ga_{\mathrm{new}} = \ga + A^*u$ is smooth.
\end{proof}

\ 

\subsection{Proof of Theorem \ref{lem:approximationintro}}
\label{sec:approx}

Recall that the \emph{principal stratum} $M^{\mathrm{princ}}\subseteq M$ is the open and dense subset given as the union of the orbits of principal isotropy type \cite{alexandrino2015lie}*{\S 3.4}. We have that $X^*=M^{\mathrm{princ}}/G$ is a smooth manifold \emph{without boundary}, and $\pi\colon M^{\mathrm{princ}}\to X^*$ is a fiber bundle \cite{alexandrino2015lie}*{Theorems 3.82 and 3.95}. Throughout, we assume that the reader is familiar with the standard terminology for Riemannian submersions and metric foliations in general, as in \cite{gw}*{Chapter 1}.

\begin{lemma}\label{lem:compactlift}
Let $X^* = M^{\mathrm{princ}}/G$ be the principal orbit space with quotient projection
$\pi\colon M^{\mathrm{princ}} \to X^*$ and $(\bar{\Phi}_t)_{t \in [0,1]}$ a smooth isotopy of $X^*$
generated by a time-dependent vector field $\bar{v}_t$. If $\bar{v}_t$ has compact support
$K \subset X^*$, then its horizontal lift $\widetilde{v}_t$ extends smoothly by zero to $M$,
generating a global $G$-equivariant diffeomorphism $\widetilde{\Phi} \in \mathrm{Diff}_G(M)$
covering $\bar{\Phi}_1$, i.e.\ such that $\pi\circ \widetilde \Phi=\bar\Phi_1\circ \pi$.
\end{lemma}
\begin{proof}
Because $M$ is compact, the rbit projection $\widehat\pi: M \to M/G$ is proper. Since $K$ is a compact subset of $X^* \subset M/G$, its preimage $\widetilde{K} := \hat{\pi}^{-1}(K)$ is a compact subset of $M$. Furthermore, because $K$ lies entirely within the principal stratum $X^*$, $\widetilde{K}$ is contained in $M^{\mathrm{princ}}$ and coincides with $\pi^{-1}(K)$. 

The horizontal lift $\widetilde{v}_t$ of $\bar{v}_t$ is smooth, $G$-invariant (since the action is isometric), and supported in $\widetilde{K}$. Because $\widetilde{K}$ is a compact subset of the open submanifold $M^{\mathrm{princ}}$, extending $\widetilde{v}_t$ by zero yields a globally smooth, $G$-invariant vector field on $M$. As $M$ is a closed manifold, this extended vector field generates a global time-one flow, which is a $G$-equivariant diffeomorphism $\widetilde{\Phi} \in \mathrm{Diff}_G(M)$. Since $\widetilde{v}_t$ is $\pi$-related to $\bar{v}_t$, we have $\pi \circ \widetilde{\Phi} = \bar{\Phi}_1 \circ \pi$, concluding the proof.
\end{proof}

\begin{lemma}\label{lem:intervallift}
Assume $X = M/G$ is homeomorphic to $[0,L]$. Let $\varphi\colon[0,L]\to[0,L]$ be a smooth strictly increasing diffeomorphism
with $\varphi(0)=0$ and $\varphi(L)=L$. Assume in addition that $\varphi$ is linear near each endpoint. Then there exists a
$G$-equivariant diffeomorphism $\widetilde\Phi\in\mathrm{Diff}_G(M)$ covering $\varphi$.
\end{lemma}
\begin{proof}
Let $\varphi_t := (1-t)\,\mathrm{id} + t\,\varphi$ be the smooth isotopy of $[0,L]$ through strictly increasing diffeomorphisms fixing $0$ and $L$. Let $v_t$ be the time-dependent vector field generating this isotopy, defined by the relation $\partial_t\varphi_t = v_t \circ \varphi_t$. 

Let $p_- = 0$ and $p_+ = L$ denote the endpoints of the interval, with respective local distance coordinates $s^-$ and $s^+$. Because $\varphi$ is linear near the endpoints, there exist constants $a^{\pm} > 0$ such that $\varphi(s^{\pm}) = a^{\pm} s^{\pm}$ on sufficiently small neighborhoods of $p_{\pm}$. Consequently, each $\varphi_t$ is also linear on these neighborhoods, and $v_t$ takes the explicit form $v_t = c^{\pm}(t)\,s^{\pm}\,\partial_{s^{\pm}}$, where $c^{\pm}(t) = \frac{a^{\pm}-1}{(1-t) + t\,a^{\pm}}$. 

Let $Q_{\pm} = \pi^{-1}(p_{\pm}) \cong G/K_{\pm}$ be the singular orbits. By the slice theorem \cite{alexandrino2015lie}*{Theorem 3.57}, a $G$-invariant tubular neighborhood of $Q_{\pm}$ is equivariantly diffeomorphic to the normal bundle $G \times_{K_{\pm}} V_{\pm}$, where $V_{\pm}$ is the slice representation equipped with a $K_{\pm}$-invariant inner product. The base coordinate on the quotient is given by $s^{\pm} = \|\xi\|$ for $\xi \in V_{\pm}$. Let $E^{\pm} = \sum_i \xi^i \partial_{\xi^i}$ denote the Euler vector field on $V_{\pm}$. The vector field $(0, E^{\pm})$ on $G \times V_{\pm}$ is $G \times K_{\pm}$-invariant and descends to a smooth $G$-invariant vector field on $G \times_{K_{\pm}} V_{\pm}$, which we also denote by $E^{\pm}$. By construction, $E^{\pm}$ is smooth across the zero section $\xi=0$ (where it vanishes) and is $\pi$-related to $s^{\pm} \partial_{s^{\pm}}$. Shrinking the tubular neighborhoods to lie within the region where $\varphi_t$ is linear, we deduce that $c^{\pm}(t)E^{\pm}$ is a smooth, $G$-invariant lift of $v_t$ that vanishes identically on $Q_{\pm}$.

Over the principal stratum $M^{\mathrm{princ}} = \pi^{-1}\bigl((0,L)\bigr)$, let $\overline{v}_t$ be the unique horizontal lift of $v_t$ with respect to the prescribed $G$-invariant metric, satisfying $\mathrm{d}\pi(\overline{v}_t) = v_t \circ \pi$. 

We glue these local lifts using a partition of unity. Choose a smooth partition of unity $\{\rho_-, \rho_0, \rho_+\}$ on $[0,L]$ subordinate to the open cover $\{[0, \epsilon), (0, L), (L-\epsilon, L]\}$, where $\epsilon$ is chosen so that $\varphi_t$ is linear on $[0, \epsilon)$ and $(L-\epsilon, L]$. Define a global time-dependent vector field on $M$ by:
\[
  \widetilde{v}_t := (\rho_- \circ \pi)\, c^-(t)E^- + (\rho_0 \circ \pi)\, \overline{v}_t + (\rho_+ \circ \pi)\, c^+(t)E^+
\]
Because $\rho_{\pm} \equiv 1$ near $p_{\pm}$, the functions $\rho_{\pm} \circ \pi$ are locally constant near the singular orbits, making them smooth everywhere. Similarly, since $\operatorname{supp}(\rho_0) \subset (0,L)$, the term $(\rho_0 \circ \pi)\,\overline{v}_t$ vanishes on a neighborhood of $Q_{\pm}$, circumventing the locus where $\overline{v}_t$ is undefined. Thus, $\widetilde{v}_t$ is a globally smooth $G$-invariant vector field on $M$.

Since each summand is $\pi$-related to $v_t$ on its respective support, linearity of the pushforward yields $\mathrm{d}\pi(\widetilde{v}_t) = v_t \circ \pi$ globally. Because $M$ is a closed manifold, $\widetilde{v}_t$ generates a global, smooth $G$-equivariant isotopy $\widetilde{\Phi}_t \in \mathrm{Diff}_G(M)$. Evaluating at time $t=1$ yields the required diffeomorphism $\widetilde{\Phi} = \widetilde{\Phi}_1$, which covers $\varphi$ by construction.
\end{proof}

A continuous map $h \colon S^1 \to S^1$, with $S^1 \cong \mathbb{R}/L\mathbb{Z}$, is defined as \emph{weakly monotone of degree one} if it admits a continuous, non-decreasing lift $H \colon \mathbb{R} \to \mathbb{R}$ satisfying $H(\theta+L) = H(\theta)+L$.

\begin{lemma}\label{lem:monotoneapprox}
\noindent\textup{\textbf{(1)}} Let $h \colon [0,L] \to [0,L]$ be a weakly monotone continuous map satisfying $h(0)=0$ and $h(L)=L$. For any $\delta>0$, there exists a smooth, strictly increasing diffeomorphism $\varphi \colon [0,L] \to [0,L]$ with $\varphi(0)=0$ and $\varphi(L)=L$, which is linear near the endpoints and satisfies $\|\varphi - h\|_\infty < \delta$.

\noindent\textup{\textbf{(2)}} Let $h \colon S^1 \to S^1$ be a continuous weakly monotone map of degree one, where $S^1 \cong \mathbb{R}/L\mathbb{Z}$. For any $\delta > 0$, there exists a smooth degree-one diffeomorphism $\varphi \colon S^1 \to S^1$ such that $\|\varphi - h\|_\infty < \delta$.
\end{lemma}
\begin{proof}
Fix a smooth mollifier $\rho_\tau \ge 0$ with $\operatorname{supp}\rho_\tau \subset [-\tau,\tau]$ and $\int_{\mathbb{R}} \rho_\tau = 1$.

\textbf{(1)} For $\lambda \in (0,1)$, define $h_\lambda := (1-\lambda)h + \lambda\,\mathrm{id}$. Then $h_\lambda(0)=0$ and $h_\lambda(L)=L$. For any $x < y$,
$$h_\lambda(y) - h_\lambda(x) = (1-\lambda)\bigl(h(y) - h(x)\bigr) + \lambda(y - x) \ge \lambda(y - x) > 0,$$
demonstrating that $h_\lambda$ is strictly increasing. Furthermore,
$$\|h_\lambda - h\|_\infty = \lambda\|\mathrm{id} - h\|_\infty \le \lambda L.$$

Extend $h$ to a map $\hat{h} \colon \mathbb{R} \to \mathbb{R}$ by setting $\hat{h}(x) = x$ for $x \notin [0,L]$. Because $h(0)=0$ and $h(L)=L$, $\hat{h}$ is continuous and non-decreasing over $\mathbb{R}$. The map $\hat{h}_\lambda := (1-\lambda)\hat{h} + \lambda\,\mathrm{id}$ naturally extends $h_\lambda$. Define the smooth function $\psi := \hat{h}_\lambda * \rho_\tau$ on $\mathbb{R}$. By linearity of convolution, $\psi = (1-\lambda)(\hat{h} * \rho_\tau) + \lambda(\mathrm{id} * \rho_\tau)$. Since $(\mathrm{id} * \rho_\tau)' \equiv 1$, and since for $x_1 < x_2$,
$$(\hat{h} * \rho_\tau)(x_2) - (\hat{h} * \rho_\tau)(x_1) = \int_{\mathbb{R}} \bigl(\hat{h}(x_2 - y) - \hat{h}(x_1 - y)\bigr)\rho_\tau(y)\,\mathrm{d}y \ge 0,$$
the function $\hat{h} * \rho_\tau$ is non-decreasing, implying $(\hat{h} * \rho_\tau)' \ge 0$. Consequently,
$$\psi' \ge \lambda > 0.$$
Thus, $\psi$ is strictly increasing, and $\psi'$ is bounded for a fixed $\tau$. Because $\hat{h}_\lambda$ is uniformly continuous on $\mathbb{R}$, $\psi \to \hat{h}_\lambda$ uniformly on $[0,L]$ as $\tau \to 0$.

To ensure the map is linear near the endpoints, define a smooth function $m \colon [0, L] \to (0,\infty)$ such that $m = \psi'$ on $[2\sigma, L-2\sigma]$, $m$ is constant on $[0,\sigma]$ and on $[L-\sigma, L]$, and $m$ is bounded uniformly by $\|\psi'\|_\infty$ on the transition intervals. By construction, $m = \psi'$ everywhere except on a set of measure $4\sigma$. Therefore, $\|m - \psi'\|_{L^1} \le 8\|\psi'\|_\infty\,\sigma$, which vanishes as $\sigma \to 0$. Define $\tilde{\varphi}(x) := \psi(0) + \int_0^x m$. Because $m > 0$, $\tilde{\varphi}$ is a smooth, strictly increasing function that is affine on $[0,\sigma]$ and $[L-\sigma, L]$. Furthermore,
$$|\tilde{\varphi}(x) - \psi(x)| = \Bigl|\int_0^x (m - \psi')\Bigr| \le \|m - \psi'\|_{L^1}.$$
Finally, define the normalized map $\varphi := L\,(\tilde{\varphi} - \tilde{\varphi}(0)) / (\tilde{\varphi}(L) - \tilde{\varphi}(0))$, which is smooth, strictly increasing, affine near the endpoints, and satisfies $\varphi(0)=0$ and $\varphi(L)=L$. As $\tau, \sigma \to 0$, the limits $\tilde{\varphi}(0) = \psi(0) \to 0$ and $\tilde{\varphi}(L) = \psi(0) + \int_0^L m \to L$ ensure the scaling factor converges to $1$, yielding $\|\varphi - \tilde{\varphi}\|_\infty \to 0$. By the triangle inequality,
$$\|\varphi - h\|_\infty \le \|\varphi - \tilde{\varphi}\|_\infty + \|m - \psi'\|_{L^1} + \|\psi - h_\lambda\|_\infty + \lambda L.$$
By selecting $\lambda$, then $\tau$, and subsequently $\sigma$ sufficiently small, the right-hand side is bounded strictly by $\delta$.

\textbf{(2)} Lift $h$ to a continuous, non-decreasing function $H \colon \mathbb{R} \to \mathbb{R}$ satisfying $H(\theta+L) = H(\theta)+L$. The difference $H - \mathrm{id}$ is $L$-periodic, and therefore bounded. Define $H_\lambda := (1-\lambda)H + \lambda\,\mathrm{id}$ and $\Phi := H_\lambda * \rho_\tau$. By identical reasoning to part (1), $\Phi = (1-\lambda)(H * \rho_\tau) + \lambda(\mathrm{id} * \rho_\tau)$, which yields $\Phi' \ge \lambda > 0$, guaranteeing that $\Phi$ is smooth and strictly increasing. Moreover, 
$$\Phi(\theta+L) = \int_{\mathbb{R}} H_\lambda(\theta+L-y)\rho_\tau(y)\,\mathrm{d}y = \Phi(\theta) + L.$$
Thus, $\Phi$ descends to a smooth degree-one diffeomorphism $\varphi \colon S^1 \to S^1$. Because $\|H_\lambda - H\|_\infty = \lambda\|H - \mathrm{id}\|_\infty < \infty$, and $\Phi \to H_\lambda$ uniformly as $\tau \to 0$, the uniform distance on the circle is bounded by the uniform distance of the respective lifts. Consequently, $\|\varphi - h\|_\infty \le \|\Phi - H\|_\infty < \delta$ for sufficiently small $\lambda$ and $\tau$.
\end{proof}

\begin{definition}\label{def:admissible}
Assume $X$ is homeomorphic to $[0,L]$ or to $S^1$. A self-map $\varphi$ of $X$ is defined as an \emph{admissible reparametrization} if:
\begin{enumerate}[label=\textup{(\roman*)}, font=\upshape]
  \item for the interval type, $\varphi \colon [0,L] \to [0,L]$ is a smooth, strictly increasing diffeomorphism satisfying $\varphi(0)=0$ and $\varphi(L)=L$, which is linear near each endpoint;
  \item for the circle type, $\varphi \colon S^1 \to S^1$ is a smooth degree-one diffeomorphism.
\end{enumerate}
\end{definition}

\begin{lemma}\label{lem:onedimapprox}
Assume $\dim X = 1$ and let $\mu$ be a finite Borel measure on $X$ that is absolutely continuous with respect to the Lebesgue measure. Let $f_1 \colon X \to \mathbb{R}$ be a continuous function, and define $D := \max_X f_1 - \min_X f_1$. For any fixed $1 \le p < \infty$ and $\epsilon > 0$, the following hold:
\begin{enumerate}[label=\textup{(\arabic*)}, font=\upshape]
  \item For every $c \in [\min_X f_1, \max_X f_1]$, there exists an admissible reparametrization $\varphi$ of $X$ satisfying $\|f_1 \circ \varphi - c\|_{L^p(\mu)} < \epsilon$.
  \item If $h \colon X \to X$ is a continuous, weakly monotone map (fixing the endpoints in the interval type, or possessing degree one in the circle type), there exists an admissible reparametrization $\varphi$ of $X$ satisfying $\|f_1 \circ \varphi - f_1 \circ h\|_{L^p(\mu)} < \epsilon$.
\end{enumerate}
\end{lemma}
\begin{proof}
Because $\mu$ is absolutely continuous with respect to the Lebesgue measure, $\mu(E_\eta) \to 0$ for any family of Borel sets $E_\eta$ that shrink to the endpoints (in the interval case) or to a single point (in the circle case).

\textbf{(1)} Fix $\epsilon' > 0$ and $\eta > 0$, to be specified subsequently. Because $X$ is connected and $f_1$ is continuous, the Intermediate Value Theorem ensures that the image of $f_1$ encompasses the interval $[\min_X f_1, \max_X f_1]$. Therefore, there exists a point $\bar{x}_0 \in X$ such that $f_1(\bar{x}_0) = c$. By continuity, there exists an open neighborhood $U_0$ containing $\bar{x}_0$ upon which the oscillation of $f_1$ is strictly less than $\epsilon'$. 

Construct an admissible reparametrization $\varphi$ and designate a Borel set $E_\eta$ as follows: for the interval type, define $E_\eta := [0,\eta] \cup [L-\eta, L]$ and choose $\varphi$ such that it maps $[\eta, L-\eta]$ into $U_0$ while remaining linear near the endpoints. For the circle type, define $E_\eta := A_\eta$ as an arc of length $\eta$, and choose a degree-one $\varphi$ that maps $X \setminus A_\eta$ into $U_0$. In both instances, $\varphi(X \setminus E_\eta) \subset U_0$. Consequently, for any $r \in X \setminus E_\eta$,
$$|f_1(\varphi(r)) - c| = |f_1(\varphi(r)) - f_1(\bar{x}_0)| < \epsilon',$$
whereas $|f_1 \circ \varphi - c| \le D$ uniformly on $E_\eta$. Integrating over $X$ yields:
$$\|f_1 \circ \varphi - c\|_{L^p(\mu)}^p \le (\epsilon')^p\,\mu(X) + D^p\,\mu(E_\eta).$$
Because $\mu(E_\eta) \to 0$ as $\eta \to 0$, fixing $\epsilon'$ sufficiently small and subsequently shrinking $\eta$ ensures the right-hand side is strictly bounded by $\epsilon^p$.

\textbf{(2)} Because $X$ is a compact metric space, $f_1$ is uniformly continuous. Fix $\delta > 0$ such that for any $z, w \in X$,
$$d_X(z,w) < \delta \implies |f_1(z) - f_1(w)| < \frac{\epsilon}{\mu(X)^{1/p}}.$$
By Lemma~\ref{lem:monotoneapprox} (utilizing part (1) for the interval type, and part (2) for the circle type), there exists an admissible reparametrization $\varphi$ satisfying $\|\varphi - h\|_\infty < \delta$. Consequently, $d_X(\varphi(r), h(r)) < \delta$ for all $r \in X$, implying $|f_1(\varphi(r)) - f_1(h(r))| < \epsilon / \mu(X)^{1/p}$. Therefore,
$$\|f_1 \circ \varphi - f_1 \circ h\|_{L^p(\mu)} \le \left( \int_X \frac{\epsilon^p}{\mu(X)}\,\mathrm{d}\mu \right)^{1/p} = \epsilon.$$
\end{proof}

\ 

We move on to the proof of Theorem \ref{lem:approximationintro}.

\ 

\begin{proof}[Proof of Theorem \ref{lem:approximationintro}]
Because $G$ acts isometrically, the complement $M \setminus M^{\mathrm{princ}}$ has measure zero, and the projection $\pi \colon M^{\mathrm{princ}} \to X^*$ is a Riemannian submersion. Let $\vartheta \colon X^* \to (0,\infty)$ be the orbit-volume function defined by $\vartheta(\bar{x}) := \mathrm{Vol}\bigl(\pi^{-1}(\bar{x})\bigr)$, which is smooth and strictly positive on $X^*$. Defining the measure $\mathrm{d}\mu := \vartheta\,\mathrm{d}V_{X^*}$, the coarea formula for Riemannian submersions yields the following identity for every $G$-invariant function $H \in L^p_G(M)$:
\begin{equation}\label{eq:Lp-quotient}
  \|H\|_{L^p(M)}^p = \|\bar{H}\|_{L^p(\mu)}^p.
\end{equation}
In particular, a $G$-equivariant diffeomorphism $\widetilde{\Phi}$ covering a base map $\bar{\Phi} \colon X^* \to X^*$ satisfies $\|f_1 \circ \widetilde{\Phi} - f_2\|_{L^p(M)} = \|\bar{f}_1 \circ \bar{\Phi} - \bar{f}_2\|_{L^p(\mu)}$, because $f_1 \circ \widetilde{\Phi}$ is $G$-invariant with the induced quotient function $\bar{f}_1 \circ \bar{\Phi}$.

\begin{enumerate}[label=\textup{(\alph*)}, font=\upshape]
  \item \textbf{Case \textup{(a):} $\dim X = k \ge 2$.} Let $\epsilon > 0$. Continuous functions are dense in $L^p(X,\mu)$, and the truncation operation $t \mapsto \max\{\min_M f_1,\, \min\{t, \max_M f_1\}\}$ is a $1$-Lipschitz map that preserves the almost-everywhere range $[\min_M f_1, \max_M f_1]$ of $\bar{f}_2$. Therefore, there exists a continuous function $\bar{f}_{2,c} \colon X \to [\min_M f_1, \max_M f_1]$ satisfying $\|\bar{f}_2 - \bar{f}_{2,c}\|_{L^p(\mu)} < \epsilon/2$. Defining $f_{2,c} := \bar{f}_{2,c} \circ \pi$, Equation \eqref{eq:Lp-quotient} implies $\|f_{2,c} - f_2\|_{L^p(M)} < \epsilon/2$.

  Fix parameters $\epsilon', \delta > 0$. By the uniform continuity of $\bar{f}_{2,c}$ on the compact space $X$ and Vitali's covering lemma, one can extract finitely many disjoint open balls $B_1, \dots, B_N \subset X^*$ upon which the oscillation $\operatorname{osc}_{B_i} \bar{f}_{2,c}$ is strictly less than $\epsilon'$, and the residual measure satisfies $\mu(X^* \setminus \bigcup_i B_i) < \delta$. Choose a reference value $c_i := \bar{f}_{2,c}(q_i)$ for some $q_i \in B_i$, ensuring $c_i \in [\min_M f_1, \max_M f_1]$ and $|\bar{f}_{2,c} - c_i| < \epsilon'$ uniformly on $B_i$. For each $i$, isolate a compact subset $K_i \subset B_i$ such that $\mu(B_i \setminus K_i) < \delta/N$.

  Because $M$ is connected, the continuous function $f_1$ attains every value in the interval $[\min_M f_1, \max_M f_1]$. Consequently, the $G$-invariant open set $\{|f_1 - c_i| < \epsilon'\}$ must intersect the dense principal stratum $M^{\mathrm{princ}}$. This intersection projects to a non-empty open set $W_i \subset X^*$ upon which $|\bar{f}_1 - c_i| < \epsilon'$. Because $k \ge 2$, we may select distinct target points $\bar{z}_i \in W_i$ and corresponding open neighborhoods $U_i \subset W_i$ containing $\bar{z}_i$ such that $\operatorname{osc}_{U_i} \bar{f}_1 < \epsilon'$. Fix points $p_i \in K_i$. By the $N$-transitivity of the identity component $\mathrm{Diff}_0(X^*)$ acting on the connected $k$-manifold $X^*$ \cites{Kankaanrinta2025, Boothby}, there exists a compactly supported diffeomorphism $\Psi$, isotopic to the identity, satisfying $\Psi(p_i) = \bar{z}_i$. Select sufficiently small open balls $B_i' \subset B_i$ containing $p_i$ such that $\Psi(B_i') \subset U_i$. 

  Let $\Xi_i$ be a diffeomorphism supported exclusively in $B_i$, isotopic to the identity, ensuring $\Xi_i(K_i) \subset B_i'$. The existence of such a map follows from standard ambient isotopy extension, as $K_i$ is a compact subset of the connected open set $B_i$, and $B_i'$ is a non-empty open sub-domain. Because the supports of the respective $\Xi_i$ are mutually disjoint, the composition
  \[
    \bar{\Phi} := \Psi \circ \Xi_N \circ \cdots \circ \Xi_1
  \]
  is a diffeomorphism isotopic to the identity, compactly supported in $X^*$, which satisfies $\bar{\Phi}(K_i) \subset \Psi(B_i') \subset U_i$ for all $i$.

  For any $\bar{x} \in K_i$, the containment $\bar{\Phi}(\bar{x}), \bar{z}_i \in U_i$ coupled with the bounds on $U_i$, $W_i$, and $B_i$ yields the pointwise bound:
  \[
    |\bar{f}_1(\bar{\Phi}(\bar{x})) - \bar{f}_{2,c}(\bar{x})| \le \operatorname{osc}_{U_i} \bar{f}_1 + |\bar{f}_1(\bar{z}_i) - c_i| + |c_i - \bar{f}_{2,c}(\bar{x})| < 3\epsilon'.
  \]
  On the complement $X^* \setminus \bigcup_i K_i$, the images of both functions reside strictly within $[\min_M f_1, \max_M f_1]$. Therefore, the integrand $|\bar{f}_1 \circ \bar{\Phi} - \bar{f}_{2,c}|^p$ is bounded from above by $C := (\max_M f_1 - \min_M f_1)^p$. Because $\mu(X^* \setminus \bigcup_i K_i) \le \delta + N(\delta/N) = 2\delta$, integrating over the domain gives:
  \[
    \|\bar{f}_1 \circ \bar{\Phi} - \bar{f}_{2,c}\|_{L^p(\mu)}^p < (3\epsilon')^p\,\mu(X^*) + 2C\delta,
  \]
  which can be made strictly less than $(\epsilon/2)^p$ by choosing $\epsilon'$ and $\delta$ sufficiently small.

  By Lemma~\ref{lem:compactlift}, the compactly supported isotopy $(\bar{\Phi}_t)_{t \in [0,1]}$ lifts to a global equivariant diffeomorphism $\widetilde{\Phi} \in \mathrm{Diff}_G(M)$ covering $\bar{\Phi}$. Because $f_1 \circ \widetilde{\Phi}$ is $G$-invariant with induced function $\bar{f}_1 \circ \bar{\Phi}$, applying the norm identity \eqref{eq:Lp-quotient} and the triangle inequality concludes the case:
  \[
    \|f_1 \circ \widetilde{\Phi} - f_2\|_{L^p(M)} \le \|\bar{f}_1 \circ \bar{\Phi} - \bar{f}_{2,c}\|_{L^p(\mu)} + \|f_{2,c} - f_2\|_{L^p(M)} < \frac{\epsilon}{2} + \frac{\epsilon}{2} = \epsilon.
  \]

  \item \textbf{Case \textup{(b):} $\dim X = 1$, interval type.} In this configuration, $X^* = (0,L)$ and $X \cong [0, L]$, with the two boundary points corresponding to non-principal (singular or exceptional) orbits. The measure $\mathrm{d}\mu = \vartheta\,\mathrm{d}r$ incorporates a smooth and strictly positive function $\vartheta$ on $(0,L)$ such that $\mu(X) < \infty$. Therefore, $\mu$ is a finite, absolutely continuous measure. We apply Lemma~\ref{lem:onedimapprox} setting $\bar{f}_1$ in the role of $f_1$, where $\min_X \bar{f}_1 = \min_M f_1$ and $\max_X \bar{f}_1 = \max_M f_1$.

  In sub-case \textup{(b.1)}, where $f_2 \equiv c \in [\min_M f_1, \max_M f_1]$, Lemma~\ref{lem:onedimapprox}(1) provides an admissible reparametrization $\varphi$ satisfying $\|\bar{f}_1 \circ \varphi - c\|_{L^p(\mu)} < \epsilon$. In sub-case \textup{(b.2)}, where $\bar{f}_2 = \bar{f}_1 \circ h$ for a continuous weakly monotone map $h \colon [0,L] \to [0,L]$ fixing the endpoints, Lemma~\ref{lem:onedimapprox}(2) guarantees an admissible $\varphi$ satisfying $\|\bar{f}_1 \circ \varphi - \bar{f}_2\|_{L^p(\mu)} < \epsilon$.

  In both sub-cases, $\varphi$ is interval-admissible. Lemma~\ref{lem:intervallift} therefore constructs $\widetilde{\Phi} \in \mathrm{Diff}_G(M)$ covering $\varphi$. Because $f_1 \circ \widetilde{\Phi}$ is $G$-invariant with induced function $\bar{f}_1 \circ \varphi$, the norm identity \eqref{eq:Lp-quotient} yields
  \[
    \|f_1 \circ \widetilde{\Phi} - f_2\|_{L^p(M)} = \|\bar{f}_1 \circ \varphi - \bar{f}_2\|_{L^p(\mu)} < \epsilon.
  \]

  \medskip
  \item \textbf{Case \textup{(c):} $\dim X = 1$, circle type.} Here $X \cong S^1$, and all orbits are principal \cite{alexandrino2015lie}*{\S 6.3}, implying $M = M^{\mathrm{princ}}$ and making the global projection $\pi \colon M \to S^1$ a Riemannian submersion. Parametrizing $S^1 = \mathbb{R}/L\mathbb{Z}$, the measure is given by $\mathrm{d}\mu = \vartheta\,\mathrm{d}\theta$ with $\vartheta$ smooth and positive, such that $\mu(X) < \infty$. As before, $\mu$ is finite and absolutely continuous.

  A circle-admissible $\varphi$ is necessarily isotopic to the identity. Specifically, for a lift $\Phi \colon \mathbb{R} \to \mathbb{R}$ of $\varphi$ satisfying $\Phi(\theta+L) = \Phi(\theta)+L$, the projections $\varphi_t$ of the linear interpolation $(1-t)\,\mathrm{id} + t\,\Phi$ construct a smooth isotopy of $S^1 = X^*$ through degree-one diffeomorphisms, connecting the identity map to $\varphi$. The generating vector field for this isotopy is trivially compactly supported on $X^* = S^1$ because the manifold itself is closed. Consequently, Lemma~\ref{lem:compactlift} ensures the existence of an equivariant lift $\widetilde{\Phi} \in \mathrm{Diff}_G(M)$ covering $\varphi$.

  Applying Lemma~\ref{lem:onedimapprox} with $\bar{f}_1$ in the role of $f_1$: sub-case \textup{(c.1)} combined with part~(1) yields an admissible $\varphi$ such that $\|\bar{f}_1 \circ \varphi - c\|_{L^p(\mu)} < \epsilon$. For sub-case \textup{(c.2)}, where $\bar{f}_2 = \bar{f}_1 \circ h$ for a continuous weakly monotone degree-one map $h \colon S^1 \to S^1$, part~(2) guarantees an admissible $\varphi$ such that $\|\bar{f}_1 \circ \varphi - \bar{f}_2\|_{L^p(\mu)} < \epsilon$. Lifting $\varphi$ precisely as described above, Equation \eqref{eq:Lp-quotient} directly dictates:
  \[
    \|f_1 \circ \widetilde{\Phi} - f_2\|_{L^p(M)} = \|\bar{f}_1 \circ \varphi - \bar{f}_2\|_{L^p(\mu)} < \epsilon. \qedhere
  \]
\end{enumerate}
\end{proof}

\ 

\subsection{Proof of Theorem \ref{theorem:GKW}}

We need the following auxiliary result.

\begin{lemma}\label{lem:dense_kernel}
The set of $G$-invariant metrics $\ga \in \mathcal{M}_G^{2,p}$ for which $\ker A_{\ga}^* = \{0\}$ is open and dense in $\mathcal{M}_G^{2,p}$. In particular, if a smooth metric $\ga_0$ satisfies $C_1 < \mathrm{scal}_{\ga_0} < C_2$, there exists a smooth metric $\ga$ arbitrarily close to $\ga_0$ in $\mathcal{M}_G^{2,p}$ such that $\ker A_{\ga}^* = \{0\}$ and $C_1 < \mathrm{scal}_{\ga} < C_2$.
\end{lemma}
\begin{proof}
We have fixed everywhere $p>n$, so that metrics in $\mathcal M_G^{2,p}$ are of class $C^1$. For $\ga\in\mathcal M_G^{2,p}$ the operator
$A_\ga^*\colon W_G^{2,p}(M)\to L_G^p(M;S^2T^*M)$ is bounded and depends continuously on $\ga$ in operator norm. Its principal symbol
$\sigma_{A_\ga^*}(\xi)u=(|\xi|_\ga^2\,\ga-\xi\otimes\xi)u$ has trace
$(n-1)|\xi|_\ga^2\,u$, hence is injective for $\xi\neq0$. Thus, $A_\ga^*$ is
overdetermined elliptic and satisfies
$\|u\|_{W^{2,p}}\le C(\|A_\ga^*u\|_{L^p}+\|u\|_{L^p})$. Together with the compact
embedding $W^{2,p}\hookrightarrow L^p$, a standard argument gives
\[
  \ker A_\ga^*=\{0\}\iff \|A_\ga^*u\|_{L^p}\ge c\,\|u\|_{W^{2,p}}\ \text{for some }c>0 .
\]
Being bounded below is stable under small operator-norm perturbations and
$\ga\mapsto A_\ga^*$ is continuous, so $\{\ga:\ker A_\ga^*=\{0\}\}$ is open.

We claim that if $\ga'\in\mathcal M_G^{2,p}$ is smooth with
$\ker A_{\ga'}^*\neq\{0\}$, then there are smooth metrics $\ga_s'=(1+su)\ga'$, with $s$ arbitrarily small, such that $\ker A_{\ga_s'}^*=\{0\}$,
$\mathrm{scal}_{\ga_s'}\to\mathrm{scal}_{\ga'}$ uniformly, and
$\|\ga_s'-\ga'\|_{W^{2,p}}\to0$.

Indeed, by Lemma~\ref{lem:marsden}(a), we have that $\mathrm{scal}_{\ga'}\equiv\kappa\ge0$ is constant
($\kappa>0$, or $\mathrm{Ric}(\ga')\equiv0$ and $\kappa=0$). For $G$-invariant
$u$ and $|s|<\|u\|_\infty^{-1}$ the conformal metric $\ga_s'=(1+su)\ga'$ lies in $\mathcal M_G^{2,p}$, and, since $\left.\tfrac{d}{ds}\right|_{0}\ga_s'=u\ga'$,
the formula for $A$ gives
\[
  \left.\tfrac{d}{ds}\right|_{s=0}\mathrm{scal}_{\ga_s'}
  =A_{\ga'}(u\ga')=-\Delta_{\ga'}(n u)+\delta_{\ga'}\delta_{\ga'}(u\ga')-u\,\mathrm{scal}_{\ga'}
  =-(n-1)\Delta_{\ga'}u-\kappa u ,
\]
where we used $\delta_{\ga'}(u\ga')=-\mathrm{d} u$ and $\delta_{\ga'}(\mathrm{d} u)=-\Delta_{\ga'}u$.

Then, $\tfrac{d}{ds}|_{s=0}\mathrm{scal}_{\ga_s'}$ is constant only if $u\in E_0\oplus E_{\kappa/(n-1)}$  (for $\kappa=0$ only $E_0$), where
$E_\lambda\subset C_G^\infty(M)$ is the $G$-invariant $\lambda$-eigenspace of
$-\Delta_{\ga'}$. These eigenspaces are finite-dimensional. Since the action has cohomogeneity at least one,
the orbit space $X$ has positive dimension, so $C^\infty_G(M)\cong C^\infty(X)$ is
infinite-dimensional. Hence, we may choose $u$ outside this finite-dimensional subspace. Then $\left.\tfrac{d}{ds}\right|_{s=0}\mathrm{scal}_{\ga_s'}$ is
non-constant, hence $\mathrm{scal}_{\ga_s'}$ is non-constant for all small
$s\neq0$. Since Lemma~\ref{lem:marsden}(a) forces any metric with nontrivial kernel to have positive constant scalar curvature or to be Ricci-flat, both excluded, $\ker A_{\ga_s'}^*=\{0\}$. Finally, $\mathrm{scal}_{\ga_s'}\to\kappa$ uniformly and
$\|\ga_s'-\ga'\|_{W^{2,p}}=|s|\,\|u\ga'\|_{W^{2,p}}\to0$. This verifies the claim.

Given $\ga_0$ and $\epsilon>0$, pick a smooth $\ga_0'$ with
$\|\ga_0'-\ga_0\|_{W^{2,p}}<\epsilon/2$. If $\ker A_{\ga_0'}^*=\{0\}$ we are done.
Otherwise, the argument in the previous paragraph furnishes $\ga_s'$ with $\ker A_{\ga_s'}^*=\{0\}$ and, for
$|s|$ small and nonzero, $\|\ga_s'-\ga_0\|_{W^{2,p}}<\epsilon$.

 Let $\ga_0$ be smooth with
$C_1<\mathrm{scal}_{\ga_0}<C_2$. If $\ker A_{\ga_0}^*=\{0\}$, take $\ga=\ga_0$. Otherwise, we apply the claim above to $\ga_0$, obtaining smooth metrics $\ga_s=(1+su)\ga_0$, and set
$\epsilon_0=\min_{x\in M}\{\mathrm{scal}_{\ga_0}(x)-C_1,\,C_2-\mathrm{scal}_{\ga_0}(x)\}>0$.
For $|s|$ small and nonzero, $\ga_s$ is smooth, $\ker A_{\ga_s}^*=\{0\}$ and
$\|\mathrm{scal}_{\ga_s}-\mathrm{scal}_{\ga_0}\|_{L^\infty}<\epsilon_0$, whence
$C_1<\mathrm{scal}_{\ga_s}<C_2$. \qedhere
\end{proof}

\ 

We now proceed with the proof of Theorem \ref{theorem:GKW}.

\

\begin{proof}[Proof of Theorem \ref{theorem:GKW}]
Throughout, $\ga$ is a smooth $G$-invariant metric, $A=F'(\ga)$, and $A^*$ is its formal adjoint \eqref{eq:adjoint}. 

\medskip
\noindent\textbf{Case \textup{(a):} $\dim X \ge 2$.}
By hypothesis $c\min_M f < \mathrm{scal}_{\ga} < c\max_M f$. If $\ker A^*=\{0\}$ we keep $\ga$ unchanged. Otherwise, Lemma~\ref{lem:dense_kernel}, applied with $C_1=c\min_M f$ and $C_2=c\max_M f$, provides a smooth $G$-invariant metric, which we still denote by $\ga$, arbitrarily $W^{2,p}$-close to the original, with $\ker A_{\ga}^*=\{0\}$ and $c\min_M f<\mathrm{scal}_{\ga}<c\max_M f$. In either case $\ker A_{\ga}^*=\{0\}$. By Lemma~\ref{lem:inversefunction}, $F$ is locally surjective at $\ga$ onto an open neighbourhood $\widetilde U\subset L^p_G(M)$ of $\mathrm{scal}_{\ga}$. Fix $\epsilon>0$ with the $L^p$-ball $B(\mathrm{scal}_{\ga},\epsilon)\subset\widetilde U$.

Apply Theorem~\ref{lem:approximationintro}(a) with the function $cf$ in the role of ``$f_1$'' and $\mathrm{scal}_{\ga}$ in the role of ``$f_2$'': since $\min_M(cf)=c\min_M f\le\mathrm{scal}_{\ga}\le c\max_M f=\max_M(cf)$ and $\dim X\ge2$, there is $\widetilde\Phi\in\mathrm{Diff}_G(M)$ with $\|cf\circ\widetilde\Phi-\mathrm{scal}_{\ga}\|_{L^p}<\epsilon$. As $f$ is $G$-invariant and $\widetilde\Phi$ is $G$-equivariant, $cf\circ\widetilde\Phi\in L^p_G(M)$, hence $cf\circ\widetilde\Phi\in\widetilde U$.

\medskip
\noindent\textbf{Cases \textup{(b) and \textup(c):} the cohomogeneity-one cases $X\cong[0,L]$ and $X\cong S^1$.}
The hypothesis provides a continuous weakly monotone $h$ with $\overline{\mathrm{scal}}_{\ga}=c\,\bar f\circ h$. In both cases $h$ is surjective: in case~(b), $h(0)=0$ and $h(L)=L$ together with weak monotonicity force $h$ to be non-decreasing, so $h([0,L])=[0,L]$ by the intermediate value theorem. In case~(c), a degree-one continuous map $S^1\to S^1$ is surjective. As $f$ is non-constant and $c>0$, surjectivity of $h$ gives $\mathrm{im}(\overline{\mathrm{scal}}_{\ga})=\mathrm{im}(c\,\bar f\circ h)=\mathrm{im}(cf)$, a non-degenerate interval. Thus $\mathrm{scal}_{\ga}$ is non-constant. 

By Lemma~\ref{lem:marsden}(a), a non-trivial kernel would force $\mathrm{scal}_{\ga}$ to be a positive constant or $\mathrm{Ric}(\ga)\equiv0$ (whence $\mathrm{scal}_{\ga}\equiv0$). Because both are excluded, we have $\ker A^*=\{0\}$. By Lemma~\ref{lem:inversefunction}, $F$ is locally surjective at $\ga$ onto a neighbourhood $\widetilde U\ni\mathrm{scal}_{\ga}$; fix $\epsilon>0$ with $B(\mathrm{scal}_{\ga},\epsilon)\subset\widetilde U$.

We apply Theorem~\ref{lem:approximationintro} with $cf$ in the role of ``$f_1$'' and $\mathrm{scal}_{\ga}$ in the role of ``$f_2$'', using the same map $h$. Since $\bar f_1=\overline{(cf)}=c\,\bar f$, the matching hypothesis reads $\bar f_2=\overline{\mathrm{scal}}_{\ga}=c\,\bar f\circ h=\bar f_1\circ h$, which is precisely sub-condition~(b.2) when $X\cong[0,L]$ (the map $h$ fixes the boundary points) and sub-condition~(c.2) when $X\cong S^1$ (the map $h$ has degree one). In either case there is $\widetilde\Phi\in\mathrm{Diff}_G(M)$ with $\|cf\circ\widetilde\Phi-\mathrm{scal}_{\ga}\|_{L^p}<\epsilon$. As $f$ is $G$-invariant and $\widetilde\Phi$ is $G$-equivariant, $cf\circ\widetilde\Phi\in L^p_G(M)$, and hence $cf\circ\widetilde\Phi\in\widetilde U$.

In both cases above we have a smooth $G$-invariant metric $\ga$ with
$\ker A^*=\{0\}$ and $cf\circ\widetilde\Phi\in\widetilde U$. By local
surjectivity (Lemma~\ref{lem:inversefunction}) there is a $G$-invariant metric
$\widetilde\ga$ near $\ga$ with
$\mathrm{scal}_{\widetilde\ga}=F(\widetilde\ga)=cf\circ\widetilde\Phi$. Since
$cf\circ\widetilde\Phi$ is smooth (as $f$ and $\widetilde\Phi$ are), the same
lemma makes $\widetilde\ga$ smooth. The metric $\widehat\ga:=c\,(\widetilde\Phi^{-1})^*\widetilde\ga$ is smooth and $G$-invariant. Using $\mathrm{scal}_{c\ga'}=c^{-1}\mathrm{scal}_{\ga'}$ for any metric $\ga'$ and constant $c>0$, together with $\mathrm{scal}_{\Psi^*\ga'}=\mathrm{scal}_{\ga'}\circ\Psi$ for any diffeomorphism $\Psi$, this finishes the proof because
\[
 \mathrm{scal}_{\widehat\ga}
 = c^{-1}\,\mathrm{scal}_{(\widetilde\Phi^{-1})^*\widetilde\ga}
 = c^{-1}\bigl(\mathrm{scal}_{\widetilde\ga}\circ\widetilde\Phi^{-1}\bigr)
 = c^{-1}\bigl((cf\circ\widetilde\Phi)\circ\widetilde\Phi^{-1}\bigr)
 = f . \qedhere
\]
\end{proof}

\

\

\section{The Yamabe problem for \texorpdfstring{$G$}{G}-invariant metrics}
\label{sec:Yamabe}

Let $c>0$. We establish the existence of a Riemannian metric $\widetilde{\ga}$ with constant scalar curvature $c$, obtained as a smooth positive $G$-invariant solution to the following PDE derived by a suitable conformal change of $\ga$:
\begin{equation}\label{eq:PDE}
-4b_n\Delta_{\ga}u + \mathrm{scal}_{\ga}u - cu^{\gamma_n} = 0,
\end{equation}
where $\gamma_n := \dfrac{n+2}{n-2}$ and $b_n :=\dfrac{n-1}{n-2}$. We approach the existence of solutions to \eqref{eq:PDE} by finding critical points (specifically, minimizers) of the functional
\begin{equation}
J(u) = 2b_n\int_M|\nabla u|_{\ga}^2 + \frac{1}{2}\int_M\mathrm{scal}_{\ga}u^2 - \frac{c}{2^*}\int_M|u|^{2^*},
\end{equation}
defined in $W^{1,2}(M)$, where $2^* := \dfrac{2n}{n-2}$. 

We assume $\mathrm{scal}_{\ga}$ to be a continuous $G$-invariant function. This ensures that $J$ is well-defined and of class $C^1$ as a functional on $W^{1,2}_G(M)$. Finding a \emph{symmetric critical point} for $J$ requires finding $u\in W^{1,2}_G(M)$ such that $\mathrm{d} J(u)(v) = 0$ for every $v\in W^{1,2}_G(M)$. By Palais' Principle of Symmetric Criticality \cite{palais1979}, because the group acts by isometries, any critical point of the restricted functional $J|_{W^{1,2}_G(M)}$ is a critical point of the full functional $J$ on $W^{1,2}(M)$. Thus, $\mathrm{d} J(u)(v)=0$ for all $v\in W^{1,2}_G(M)$ implies the same equality holds for all test functions $v\in W^{1,2}(M)$. The following compactness result is central to the argument:

\begin{theorem}[Hebey--Vaugon; see {\cite{hebey1996sobolev}*{Thm.~5.6}}]\label{thm:largerembedings}
Let $(M,\ga)$ be a closed connected Riemannian manifold of dimension $n$ with an effective isometric action by a compact Lie group $G$. If the minimal orbit dimension $d = \min_{x\in M}\dim Gx \geq 1$, then for any $p \in [1,n)$ there exists a $p_0 > \frac{np}{n-p} =: p^*$ such that $W^{1,p}_G(M)$ embeds compactly in $L^q_G(M)$ for every $1 \leq q \leq p_0$. 
\end{theorem}

In our variational setup, $p=2$. Theorem \ref{thm:largerembedings} ensures that the embedding $W^{1,2}_G(M) \hookrightarrow L^{2^*}_G(M)$ is compact.

\begin{lemma}\label{lem:criticalpoint}
    Assume that $\mathrm{scal}_{\ga}$ is a continuous function with $\min_M\mathrm{scal}_{\ga} \geq 0$, and let $c > 0$. Given $\epsilon > 0$, consider the set
    \begin{equation}
        \Sigma := \left\{u \in W^{1,2}_G(M) : \frac{c}{2^*}\int_M |u|^{2^*} = \epsilon \right\}.
    \end{equation}
    Then, for the functional
    \begin{equation}
     J(u) = 2b_n\int_M|\nabla u|_{\ga}^2 + \frac{1}{2}\int_M\mathrm{scal}_{\ga}u^2
       - \frac{c}{2^*}\int_M |u|^{2^*},
    \end{equation}
    it holds that:
    \begin{enumerate}
        \item $\Sigma$ is non-empty and weakly closed, and $J|_{{\Sigma}}$ is weakly lower semicontinuous;
        \item $J|_{\Sigma}$ is coercive.
    \end{enumerate}
\end{lemma}

\begin{proof}
\begin{enumerate}
        \item Evaluating the integral on an appropriate constant shows that $\Sigma$ is non-empty. Let $\{u_m\}\subset\Sigma$ with $u_m\rightharpoonup u$ in $W^{1,2}_G(M)$. As this
space is a closed subspace of the reflexive $W^{1,2}(M)$, it is weakly closed, so
$u\in W^{1,2}_G(M)$. By Theorem~\ref{thm:largerembedings} the embedding
$W^{1,2}_G(M)\hookrightarrow L^{2^*}_G(M)$ is compact, hence $u_m\to u$ strongly
in $L^{2^*}_G(M)$ and $\frac{c}{2^*}\int_M|u_m|^{2^*}\to\frac{c}{2^*}\int_M|u|^{2^*}$,
giving $u\in\Sigma$. Thus $\Sigma$ is weakly closed.

        To establish weak lower semicontinuity, we observe that weak convergence in $W^{1,2}_G(M)$ implies:
        $$\liminf_{m\to\infty}\int_M|\nabla u_m|^2_{\ga} \geq \int_M|\nabla u|^2_{\ga}.$$
        Using the strong convergence of $\{u_m\}$ in $L^2_G(M)$ and the continuity of $\mathrm{scal}_{\ga}$, we have:
        $$\lim_{m\to\infty} \int_M \mathrm{scal}_{\ga}u_m^2 = \int_M \mathrm{scal}_{\ga}u^2.$$
        Combining this with the strong convergence of the $L^{2^*}$ term established above, we conclude:
        $$\liminf_{m\to\infty}J(u_m) \geq 2b_n\int_M|\nabla u|^2_{\ga} +\frac{1}{2}\int_M\mathrm{scal}_{\ga}u^2-\frac{c}{2^*}\int_M|u|^{2^*} = J(u).$$
        
    \item For coercivity, observe that on the constrained set $\Sigma$, we have $\int_M |u|^{2^*} = \epsilon 2^* / c$. By the compactness of $M$, Hölder's inequality gives:
        $$\|u\|_{L^2(M)}^2 \leq \mathrm{Vol}(M)^{\frac{2}{n}} \|u\|_{L^{2^*}(M)}^2.$$
        Consequently, the $L^2$-norm of $u$ is uniformly bounded on $\Sigma$.
        
        Substituting the bounds into $J$, we obtain:
        \begin{align}\label{eq:inequality1}
            J(u) &\geq 2b_n\int_M|\nabla u|_{\ga}^2 + \frac{\min_M\mathrm{scal}_{\ga}}{2}\int_Mu^2 - \epsilon.
        \end{align}
        Because $\min_M \mathrm{scal}_{\ga} \geq 0$, 
        $$J(u) \geq 2b_n\int_M|\nabla u|_{\ga}^2 - \epsilon.$$
        As $\|u\|_{W_G^{1,2}(M)} \to \infty$ within $\Sigma$, we have $\|\nabla u\|_2^2 \to \infty$, and thus $J(u) \to \infty$. This proves that $J|_{\Sigma}$ is coercive. \qedhere
    \end{enumerate}
\end{proof}

\begin{proposition}\label{thm:invariantkazdan}
Let $\ga$ be a smooth $G$-invariant metric with $\mathrm{scal}_{\ga} \ge 0$ and $\mathrm{scal}_{\ga} \not\equiv 0$. Then every constant $c > 0$ is realized as the constant scalar curvature of a smooth $G$-invariant metric on $M$.
\end{proposition}
\begin{proof}
The strategy is to minimize the functional $J$ subject to the constraint defined by $\Sigma$. The resulting Lagrange multiplier yields a solution with constant scalar curvature $c'>0$, which is subsequently rescaled to $c$ via a homothety. By Lemma~\ref{lem:criticalpoint}, $J|_{\Sigma}$ is coercive and weakly lower semicontinuous on the weakly closed set $\Sigma$, ensuring it attains its infimum at some $u\in\Sigma$. Because $u \in W^{1,2}_G(M)$ implies $|u| \in W^{1,2}_G(M)$ with $|\nabla|u|| = |\nabla u|$ almost everywhere, replacing $u$ by $|u|$ leaves $\int_M|\nabla u|_{\ga}^2$, $\int_M\mathrm{scal}_{\ga} u^2$, and $\int_M|u|^{2^*}$ unchanged. Thus, $J(|u|) = J(u)$ and $|u| \in \Sigma$. We may therefore assume $u \ge 0$.

The functional $I(u) = \frac{c}{2^*}\int_M|u|^{2^*}$ is of class $C^1$ on $W^{1,2}_G(M)$ with derivative $\mathrm{d}I(u)(v) = c\int_M|u|^{2^*-2}u\,v$, yielding $\mathrm{d}I(u)(u) = 2^*\epsilon > 0$. Consequently, $\epsilon$ is a regular value, and $\Sigma = I^{-1}(\epsilon)$ is a smooth codimension-one submanifold of $W^{1,2}_G(M)$. By the Lagrange Multiplier Theorem, there exists $\lambda \in \mathbb{R}$ such that $\mathrm{d}J(u)(v) = \lambda\,\mathrm{d}I(u)(v)$ for all $v \in W^{1,2}_G(M)$. By the Principle of Symmetric Criticality, this equality extends to all $v \in W^{1,2}(M)$. Because $u \ge 0$, we have $|u|^{2^*-2}u = u^{\gamma_n}$, expanding the equation to:
$$4b_n\int_M \langle \nabla u,\nabla v\rangle_{\ga} + \int_M\mathrm{scal}_{\ga}uv = (1+\lambda)\,c\int_M u^{\gamma_n}v.$$
This demonstrates that $u$ weakly solves Equation \eqref{eq:PDE} with the prescribed constant replaced by $c' = (1+\lambda)c$.

Testing the equation against $v=u$ yields:
$$4b_n\int_M |\nabla u|_{\ga}^2 + \int_M\mathrm{scal}_{\ga}u^2 = (1+\lambda)\,c\int_M u^{2^*}.$$
The left-hand side is strictly positive: if $\nabla u \equiv 0$, then $u$ must be a positive constant (since $u \ge 0$ and $u \not\equiv 0$), making the second term strictly positive because $\mathrm{scal}_{\ga} \ge 0$ and $\mathrm{scal}_{\ga} \not\equiv 0$. Otherwise, $\int_M|\nabla u|_{\ga}^2 > 0$. Because $\int_M u^{2^*} > 0$ and $c > 0$, it follows that $1+\lambda > 0$, which ensures $c' = (1+\lambda)c > 0$.

We must verify that $u$ is strictly positive and smooth. By Theorem~\ref{thm:largerembedings}, $W^{1,2}_G(M)$ embeds compactly into $L^{p_0}_G(M)$ for some $p_0 > 2^*$, placing $u \in L^{p_0}(M)$. This implies $c'u^{\gamma_n} \in L^{p_0/\gamma_n}(M)$, where $p_0/\gamma_n > 2^*/\gamma_n = \frac{2n}{n+2} > 1$. Because $\mathrm{scal}_{\ga}u \in L^{p_0}(M) \subset L^{p_0/\gamma_n}(M)$, elliptic regularity dictates that $u \in W^{2,\,p_0/\gamma_n}(M)$. One iteration of elliptic regularity followed by Sobolev embedding maps a function $u \in L^{p}$ to $u \in L^{p^{\sharp}}$, where $\frac{1}{p^{\sharp}} = \frac{\gamma_n}{p} - \frac{2}{n}$. Defining $a_k = 1/p_k$, this produces the recursion $a_{k+1} - \frac{1}{2^*} = \gamma_n\bigl(a_k - \frac{1}{2^*}\bigr)$. Given $a_0 = 1/p_0 < 1/2^*$ and $\gamma_n > 1$, the sequence $a_k$ diverges to $-\infty$. After finitely many iterations, we conclude $u \in L^\infty(M)$.

Because $u \in L^\infty(M)$ and the coefficients are bounded, $\Delta_{\ga}u \in L^\infty(M)$, which implies $u \in W^{2,p}(M)$ for all $p < \infty$, and particularly $u \in C^{1,\alpha}_G(M)$. Being non-negative and non-trivial, the continuous function $u$ satisfies:
$$-4b_n\Delta_{\ga}u + \bigl(\mathrm{scal}_{\ga} - c'\,u^{\gamma_n-1}\bigr)u = 0.$$
The zeroth-order coefficient $x \mapsto \mathrm{scal}_{\ga}(x) - c'\,u(x)^{\gamma_n-1}$ belongs to $L^\infty(M)$. The strong maximum principle \cite{aubinbook}*{Proposition 3.75, p.98} requires $u > 0$ strictly everywhere on $M$. Because $M$ is compact, $u$ is bounded below by a strictly positive constant. Consequently, the map $t \mapsto t^{\gamma_n}$ is smooth on the strictly positive interval $[\min_M u, \max_M u] \subset (0,\infty)$, ensuring that $c'u^{\gamma_n}$ shares the exact regularity of $u$. Since $\ga$ and $\mathrm{scal}_{\ga}$ are smooth, elliptic bootstrapping iteratively guarantees $u \in C^\infty_G(M)$.

Finally, the conformal modification $\widetilde{\ga} = u^{\frac{4}{n-2}}\ga$ defines a $G$-invariant metric exhibiting constant scalar curvature $c' > 0$. Applying the constant scaling $k = c/c'$, the metric $\widehat{\ga} = k\widetilde{\ga}$ achieves the prescribed constant scalar curvature $c$.
\end{proof}

\begin{corollary}\label{cor:invariantkazdhan}
 Assume that $M$ carries no $G$-invariant metric of strictly positive scalar curvature. If $\ga$ is a smooth $G$-invariant metric with $\mathrm{scal}_{\ga}\geq 0$, then $\mathrm{scal}_{\ga} \equiv 0$.
\end{corollary}
\begin{proof}
If $\mathrm{scal}_{\ga}\not\equiv 0$, Proposition~\ref{thm:invariantkazdan} would produce a $G$-invariant metric of positive constant scalar curvature, contradicting the hypothesis. Hence $\mathrm{scal}_{\ga}\equiv 0$.
\end{proof}

\begin{proposition}\label{prop:auxiliandonofim}
Assume that $M$ carries no $G$-invariant metric of strictly positive scalar curvature. If there exists a smooth $G$-invariant metric $\ga$ with $\mathrm{scal}_{\ga} \geq 0$, then $\ga$ is Ricci-flat.
\end{proposition}
\begin{proof}
Let $\ga$ be a smooth $G$-invariant metric on $M$ with $\mathrm{scal}_{\ga} \geq 0$. Because $M$ does not carry a $G$-invariant metric of strictly positive scalar curvature, Corollary \ref{cor:invariantkazdhan} establishes that $\mathrm{scal}_{\ga} \equiv 0$ on $M$. Assume, for the sake of contradiction, that $\ga$ is not Ricci-flat. Since $\mathrm{scal}_{\ga} \equiv 0$ and $\mathrm{Ric}(\ga) \neq 0$, Lemma \ref{lem:marsden}(a) implies $\ker A_{\ga}^* = \{0\}$.

Because $\ga$ is smooth, Lemma \ref{lem:inversefunction} ensures that the scalar curvature operator $F$ is locally surjective at $\ga$: there exists an $L^p_G(M)$-neighborhood $\widetilde U$ of $\mathrm{scal}_{\ga}\equiv 0$ wherein every function is realized as the scalar curvature of a smooth $G$-invariant metric near $\ga$. For any constant $K > 0$, the norm $\|K\|_{L^p(M)} = K\,\mathrm{Vol}(M)^{1/p}$ tends to $0$ as $K \to 0$. Therefore, for $K$ sufficiently small, the positive constant function $K$ belongs to $\widetilde U$. Thus, there exists a smooth $G$-invariant metric $\widetilde\ga$ with $\mathrm{scal}_{\widetilde\ga} \equiv K > 0$, contradicting the hypothesis. 
\end{proof}

\begin{lemma}\label{lem:negativetotalscal}
Suppose the cohomogeneity satisfies $k:= \dim(M/G)\ge 2$. Then $M$ admits a smooth
$G$-invariant metric $\ga_1$ with $\displaystyle\int_M \mathrm{scal}_{\ga_1}\,\mathrm{d}\mu_{\ga_1} < 0$.
\end{lemma}
\begin{proof}
Fix a principal orbit $F = G\cdot p \cong G/H$ and set $m:=\dim F=n-k\ge 1$.
The Slice Theorem \cite{alexandrino2015lie}*{Theorem 3.57} gives a $G$-equivariant
diffeomorphism from a $G$-invariant tube $U$ of $F$ onto $F\times D^k$, with $G$
acting by left translations on $F$ and trivially on $D^k$. Fix a $G$-invariant metric $g_F$ on $F$. Because $F$ is a homogeneous space, $\ga_F$ has constant scalar curvature $\kappa$. Fix a smooth $G$-invariant metric $\ga_0$ on
$M$ that restricts to the product $h_0\oplus g_F$ on $U$, where $h_0$ is a flat metric on
$D^k$.

For a smooth metric $h$ on $D^k$ and a smooth $\phi\colon D^k\to(0,\infty)$, both equal to
$h_0$, resp.\ $\equiv 1$, near $\partial D^k$, set $\ga_1:=h\oplus\phi^2 g_F$ on $U$ and
$\ga_1:=\ga_0$ on $M\setminus U$. By the
warped-product formula \cite{besse1987einstein}*{Proposition 9.106},
\[
\mathrm{scal}_{\ga_1}=\mathrm{scal}_h+\kappa\,\phi^{-2}-2m\frac{\Delta_h\phi}{\phi}
-m(m-1)\frac{|\nabla\phi|_h^2}{\phi^2}\qquad\text{on }U,
\]
and $\mathrm{d}\mu_{\ga_1}=\phi^m\,\mathrm{d}\mu_h\,\mathrm{d}\mu_{g_F}$. Because $\phi$ is constant near
$\partial D^k$, integration by parts yields no boundary
contribution, and
\begin{equation}\label{eq:totalscalformula}
\int_U \mathrm{scal}_{\ga_1}\,\mathrm{d}\mu_{\ga_1}
=\mathrm{Vol}(F,g_F)\Big[\underbrace{\int_{D^k}\!\mathrm{scal}_h\,\phi^m\,\mathrm{d}\mu_h}_{(I)}
+\underbrace{\kappa\!\int_{D^k}\!\phi^{m-2}\,\mathrm{d}\mu_h}_{(II)}
+\underbrace{m(m-1)\!\int_{D^k}\!|\nabla\phi|_h^2\,\phi^{m-2}\,\mathrm{d}\mu_h}_{(III)}\Big].
\end{equation}
Because $\ga_1=\ga_0$ outside $U$,
\begin{equation}\label{eq:globalsplit}
\int_M \mathrm{scal}_{\ga_1}\,\mathrm{d}\mu_{\ga_1}
=\int_U \mathrm{scal}_{\ga_1}\,\mathrm{d}\mu_{\ga_1}+C_0,\qquad
C_0:=\int_{M\setminus U}\mathrm{scal}_{\ga_0}\,\mathrm{d}\mu_{\ga_0},
\end{equation}
where $C_0$ is a constant independent of $(h,\phi)$. It therefore suffices to make $(I)$
diverge to $-\infty$ while $(II)$ and $(III)$ remain bounded.

\emph{Case $k\ge 3$.} Take $\phi\equiv 1$, so $(III)=0$ and
$(I)+(II)=\int_{D^k}\mathrm{scal}_h\,\mathrm{d}\mu_h+\kappa\,\mathrm{Vol}(D^k,h)$. It therefore
suffices to produce, for each $\Lambda>0$, a metric $h$ on $D^k$ that equals $h_0$ near
$\partial D^k$, has bounded volume $\mathrm{Vol}(D^k,h)\le \mathrm{Vol}(D^k,h_0)+1$, and satisfies
$\int_{D^k}\mathrm{scal}_h\,\mathrm{d}\mu_h\le-\Lambda$. For $k\ge 3$ such metrics exist by Lohkamp's scalar-curvature flexibility theorem
\cite{Lohkamp1999}*{Theorem 1}. Applied on $\operatorname{int}(D^k)$ with reference
metric $h_0$ (so $\mathrm{scal}_{h_0}\equiv0$), to a smooth $f\le0$ with
$f\equiv-2\Lambda/\mathrm{Vol}(D^k,h_0)$ on a region $V\Subset\operatorname{int}(D^k)$ of
$h_0$-volume $\ge\tfrac23\mathrm{Vol}(D^k,h_0)$, $f<0$ on an open $U\Supset V$ with
$U\Subset\operatorname{int}(D^k)$, and $f\equiv0$ off $U$, it produces, for each $\varepsilon>0$,
a metric $h$ with $h\equiv h_0$ near $\partial D^k$, $\mathrm{scal}_h\le f\le0$ on $D^k$, and
$\|h-h_0\|_{C^0}<\varepsilon$. For $\varepsilon$ small, $\mathrm{Vol}(D^k,h)\le\mathrm{Vol}(D^k,h_0)+1$
and $\mathrm{Vol}(V,h)\ge\tfrac12\mathrm{Vol}(D^k,h_0)$, whence
$\int_{D^k}\mathrm{scal}_h\,\mathrm{d}\mu_h\le\int_V\mathrm{scal}_h\,\mathrm{d}\mu_h\le-\Lambda$.

Consequently, $(I)+(II)\le-\Lambda+|\kappa|\,(\mathrm{Vol}(D^k,h_0)+1)$, and \eqref{eq:globalsplit}
gives $\int_M\mathrm{scal}_{\ga_1}\,\mathrm{d}\mu_{\ga_1}<0$ for $\Lambda$ large.

\emph{Case $k=2$.} Here, Gauss-Bonnet theorem implies that $\int_{D^2}\mathrm{scal}_h\,\mathrm{d}\mu_h=0$ for every $h$ equal to $h_0$
near $\partial D^2$. Write $h=e^{2u}h_0$ with $u$ smooth and \emph{compactly supported} in the interior of
$D^2$, so that automatically $h=h_0$ near $\partial D^2$. In two dimensions
\[
\mathrm{scal}_h\,\mathrm{d}\mu_h=-2\,\Delta_0 u\,\mathrm{d}\mu_0,\qquad
|\nabla\phi|_h^2\,\mathrm{d}\mu_h=|\nabla\phi|_{h_0}^2\,\mathrm{d}\mu_0.
\]
Using that $u$ and $\phi-1$ vanish near
$\partial D^2$, we get
\[
(I)=-2\!\int_{D^2}\!\phi^m\,\Delta_0 u\,\mathrm{d}\mu_0
   =-2\!\int_{D^2}\! u\,\Delta_0(\phi^m)\,\mathrm{d}\mu_0,\qquad
(III)=m(m-1)\!\int_{D^2}\!|\nabla\phi|_{h_0}^2\,\phi^{m-2}\,\mathrm{d}\mu_0,
\]
In particular, $(III)$ depends only on $\phi$ (not on $u$).

Choose once and for all a smooth radial function $\phi\colon D^2\to[\phi_*,1]$ where
\begin{itemize}
\item $\phi$ is non-constant,
    \item $0<\phi_*<1$,
    \item $\phi$ is equal to  $1$ near $\partial D^2$.
\end{itemize}
Then, $\phi^m$ is non-constant and
$\int_{D^2}\Delta_0(\phi^m)\,\mathrm{d}\mu_0=0$, so $\Delta_0(\phi^m)<0$ on a nonempty open set
$W\subset D^2$. Fix a smooth non-negative and non-identically zero function $w$ supported in $W$. Set
\[
u=u_\lambda:=-\lambda\,w,\qquad \lambda>0 .
\]
With this choice:
\begin{itemize}
\item $(I)=-2\!\int u_\lambda\,\Delta_0(\phi^m)\,\mathrm{d}\mu_0
        =2\lambda\!\int_W w\,\Delta_0(\phi^m)\,\mathrm{d}\mu_0=-2\lambda\,\beta$, where
        $\beta:=\int_W w\,|\Delta_0(\phi^m)|\,\mathrm{d}\mu_0>0$. Thus $(I)\to-\infty$ as
        $\lambda\to\infty$.
\item Since $u_\lambda\le 0$ we have $e^{2u_\lambda}\le 1$, so
        $|(II)|=\Big|\kappa\!\int\phi^{m-2}e^{2u_\lambda}\,\mathrm{d}\mu_0\Big|
        \le|\kappa|\,\phi_*^{-|m-2|}\,\mathrm{Vol}(D^2,h_0)=:B_2$,
        a bound independent of $\lambda$.
\item $(III)=m(m-1)\!\int|\nabla\phi|_{h_0}^2\phi^{m-2}\,\mathrm{d}\mu_0=:B_3$ is independent of
        $\lambda$.
\end{itemize}
Therefore, by \eqref{eq:totalscalformula} and \eqref{eq:globalsplit},
\[
\int_M\mathrm{scal}_{\ga_1}\,\mathrm{d}\mu_{\ga_1}
\le \mathrm{Vol}(F,g_F)\big(-2\lambda\beta+B_2+B_3\big)+C_0
\xrightarrow[\ \lambda\to\infty\ ]{}-\infty. \qedhere
\]
\end{proof}

\begin{theorem}\label{thm:negativeexistence}
Suppose $M$ admits a smooth $G$-invariant metric $\ga_1$ such that $\int_M\mathrm{scal}_{\ga_1}\,\mathrm{d}\mu_{\ga_1}<0$. Then there exists a smooth $G$-invariant metric $\widehat{\ga}$ on $M$
such that $\mathrm{scal}_{\widehat{\ga}} \equiv -1$.
\end{theorem}
\begin{proof}
For the metric given in the hypothesis, we have
$\int_M\mathrm{scal}_{\ga_1}\,\mathrm{d}\mu_{\ga_1}<0$. We minimize the $G$-invariant Yamabe functional
$Q_{\ga_1} : W^{1,2}_G(M) \setminus \{0\} \to \mathbb{R}$,
\[
Q_{\ga_1}(u) = \frac{\int_M \bigl( 4b_n |\nabla u|_{\ga_1}^2 + \mathrm{scal}_{\ga_1} u^2 \bigr) \mathrm{d}\mu_{\ga_1}}{\left( \int_M |u|^{2^*} \mathrm{d}\mu_{\ga_1} \right)^{2/2^*}},
\qquad \nu := \inf_{u\neq 0} Q_{\ga_1}(u).
\]
Testing on $u\equiv 1$ gives
\[
\nu \le \mathrm{Vol}(M, \ga_1)^{-2/2^*} \int_M \mathrm{scal}_{\ga_1}\,\mathrm{d}\mu_{\ga_1} < 0.
\]
Let $\{u_m\} \subset W^{1,2}_G(M)$ be a minimizing sequence with $\|u_m\|_{L^{2^*}} = 1$ and
$Q_{\ga_1}(u_m) \to \nu$. Because $\nu < 0$, for $m$ large
\[
4b_n \!\int_M\! |\nabla u_m|_{\ga_1}^2 \mathrm{d}\mu_{\ga_1} + \!\int_M\! \mathrm{scal}_{\ga_1} u_m^2 \,\mathrm{d}\mu_{\ga_1} < 0
\implies 4b_n \|\nabla u_m\|_{L^2}^2 \le \max_M |\mathrm{scal}_{\ga_1}|\, \|u_m\|_{L^2}^2.
\]
By Hölder's inequality, $\|u_m\|_{L^2}^2 \le \mathrm{Vol}(M, \ga_1)^{2/n}\|u_m\|_{L^{2^*}}^2 = \mathrm{Vol}(M, \ga_1)^{2/n}$,
so $\{u_m\}$ is bounded in $W^{1,2}_G(M)$. Passing to a subsequence, $u_m \rightharpoonup \widehat{u}$
weakly in the reflexive space $W^{1,2}_G(M)$. Theorem~\ref{thm:largerembedings} implies that $u_m \to \widehat{u}$ strongly in
$L^{2^*}_G(M)$ and $\|\widehat{u}\|_{L^{2^*}} = 1$. Strong $L^{2^*}$ convergence gives strong $L^2$
convergence and, along a further subsequence, pointwise a.e.\ convergence. Replacing $\widehat u$ by
$|\widehat u|$ we may assume $\widehat{u} \ge 0$.

By weak lower semicontinuity of $\|\nabla\,\cdot\,\|_{L^2}$ and strong $L^2_G(M)$ convergence, the
numerator of $Q_{\ga_1}$ is weakly lower semicontinuous, so
$Q_{\ga_1}(\widehat{u}) \le \liminf_{m} Q_{\ga_1}(u_m) = \nu$. Thus, $\widehat{u}$ minimizes
$Q_{\ga_1}$. By Palais' Principle of Symmetric Criticality, $\widehat{u}$ is a critical point on the
full space $W^{1,2}(M)$, hence a weak solution of
\[
-4b_n \Delta_{\ga_1} \widehat{u} + \mathrm{scal}_{\ga_1} \widehat{u} = \nu\, \widehat{u}^{\gamma_n}.
\]
Since $\widehat{u} \ge 0$, $\|\widehat{u}\|_{L^{2^*}} = 1$ and $\nu < 0$, the bootstrap of Proposition~\ref{thm:invariantkazdan} gives $\widehat u\in C^{1,\alpha}_G(M)$. The strong maximum principle then yields $\widehat u>0$ on $M$, and since $\widehat u>0$,
further elliptic bootstrapping gives $\widehat u\in C^\infty_G(M)$. The conformal metric $\ga_2 = \widehat{u}^{4/(n-2)} \ga_1$
is a smooth $G$-invariant metric, and by \eqref{eq:PDE} we have $\mathrm{scal}_{\ga_2} \equiv \nu < 0$.
Setting $\widehat{\ga} = (-\nu)\,\ga_2$ gives
\[
\mathrm{scal}_{\widehat{\ga}} = (-\nu)^{-1}\,\mathrm{scal}_{\ga_2} = (-\nu)^{-1}\nu = -1. \qedhere
\]
\end{proof}

\ 

Lemma \ref{lem:dichotomy} provides a characterization for totally $G$-positive pairs $(M, G)$. To establish the corresponding proof, we must recall established geometric properties of manifolds equipped with cohomogeneity one actions. These properties are synthesized in Lemma \ref{lem:smooth}. 

Following \cites{Mostert, Bredon, alexandrino2015lie, GZ, Ziller}, when the action is of cohomogeneity one, the orbit space is homeomorphic to either $M/G \cong S^1$ (in which case all orbits are principal) or a closed interval $[0, L]$. In the interval case, there are two non-principal orbits $G/K_-$ and $G/K_+$ over the endpoints, encoded by a group diagram $H \subset K_\pm \subset G$. Writing $m := \dim(G/H) = n-1$ for the dimension of the principal orbits, we have $K_\pm/H \cong S^{\ell_\pm}$ with $\ell_\pm := \dim(K_\pm/H) \ge 0$ and $\dim(G/K_\pm) = m - \ell_\pm$. A non-principal orbit $G/K$ is defined as \emph{singular} if $\dim(G/K) < m$, which is equivalent to $\ell := \dim(K/H) \ge 1$. It is defined as \emph{exceptional} if $\dim(G/K) = m$, which is equivalent to $\ell = 0$; that is, the identity connected components for $K$ and $H$ coincide and $K/H \cong S^0$. 

\begin{definition}\label{def:isotropy-irreducible}
For the remainder of this section, we denote a principal orbit by $P = G/H$, where $H$ is a principal isotropy group, and write $\mathfrak{m} := T_{eH}P \cong \mathfrak{g}/\mathfrak{h}$ for the isotropy module, equipped with the linear isotropy representation of $H$. We say that $(M, G)$ has an \emph{isotropy irreducible principal orbit} if, for a principal orbit $P$, the representation of $H$ on $\mathfrak{m}$ is irreducible over $\mathbb{R}$. 
\end{definition}

\begin{lemma}\label{lem:smooth}
Assume $(M,G)$ is of cohomogeneity one with an isotropy irreducible principal orbit $P=G/H$ and contains no zero-dimensional orbits. Then $M/G \cong S^1$ or $M/G \cong [0,L]$, and the following properties hold:
\begin{itemize}
  \item[$(\mathrm{i})$] if $M/G \cong [0,L]$, both endpoints correspond to exceptional orbits;
  \item[$(\mathrm{ii})$] the space of $G$-invariant symmetric $2$-tensors on $P$ is one-dimensional. Up to scaling, $P$ admits a unique $G$-invariant Riemannian metric $\ga_P$;
  \item[$(\mathrm{iii})$] every $G$-invariant metric on the principal stratum takes the form $\ga = \mathrm{d}r^2 + \phi(r)^2\,\ga_P$, where $\phi$ is a smooth positive function and $r$ is an arc-length coordinate on $M/G$;
  \item[$(\mathrm{iv})$] the function $\phi$ is $L$-periodic if $M/G \cong S^1$. If $M/G \cong [0, L]$, $\phi$ extends to a smooth even function across each endpoint, satisfying $\phi'(0) = \phi'(L) = 0$, and the exceptional orbits are totally geodesic.
\end{itemize}
\end{lemma}
\begin{proof}
\begin{itemize}
    \item[(i)] By standard theory \cites{Mostert,Bredon,alexandrino2015lie}, the orbit space $M/G$ is a compact connected $1$-manifold, therefore homeomorphic to $S^1$ or $[0,L]$. If $M/G \cong S^1$, all orbits are principal. If $M/G \cong [0,L]$, there are two non-principal orbits $G/K_-$ and $G/K_+$. By the Slice Theorem \cite{alexandrino2015lie}*{Theorem 3.57}, the slice representation of $K_\pm$ acts transitively on the unit normal sphere, implying $K_\pm/H \cong S^{\ell_\pm}$ where $\ell_\pm = \dim(K_\pm/H)$ and $\dim(G/K_\pm) = m - \ell_\pm$.

    Assume that $G/K_-$ is singular, meaning $\ell_- \ge 1$. Choose $\operatorname{Ad}(H)$-invariant complements (orthogonal with respect to an $\operatorname{Ad}(K_-)$-invariant inner product on $\mathfrak{g}$), such that
    \[
    \mathfrak{g} = \mathfrak{k}_- \oplus \mathfrak{n}_-,\qquad \mathfrak{k}_- = \mathfrak{h} \oplus \mathfrak{p}_-.
    \]
    As $\operatorname{Ad}(H)$-modules, we obtain the decomposition
    \[
    \mathfrak{m} \cong \mathfrak{g}/\mathfrak{h} = \mathfrak{p}_- \oplus \mathfrak{n}_-,\qquad \mathfrak{p}_- \cong T_{eH}(K_-/H),\quad \mathfrak{n}_- \cong T_{eK_-}(G/K_-).
    \]
    Here, $\dim\mathfrak{p}_- = \ell_- \ge 1$. Additionally, $\dim\mathfrak{n}_- = \dim(G/K_-) = m - \ell_- \ge 1$, where the strict inequality holds because the action has no zero-dimensional orbits. Consequently, $\mathfrak{m}$ admits a nontrivial $\operatorname{Ad}(H)$-invariant decomposition into two positive-dimensional submodules, contradicting the hypothesis of isotropy irreducibility. Therefore, $\ell_\pm = 0$ and both non-principal orbits are exceptional.

    \item[(ii)] Fix an $\operatorname{Ad}(H)$-invariant inner product $\langle\cdot,\cdot\rangle$ on $\mathfrak{m}$. Every invariant symmetric bilinear form can be expressed as $b(x,y) = \langle Ax,y\rangle$ for a unique endomorphism $A \in \operatorname{End}_H(\mathfrak{m})$ satisfying $A^\ast = A$. Being self-adjoint with respect to a positive-definite inner product, $A$ is diagonalizable over $\mathbb{R}$. Each eigenspace $\ker(A - \lambda\,\mathrm{Id})$ is $H$-invariant because $A$ is $H$-equivariant. By irreducibility, each eigenspace is either $\{0\}$ or the entirety of $\mathfrak{m}$. Thus $A = \lambda\,\mathrm{Id}$ and $b = \lambda\langle\cdot,\cdot\rangle$; demonstrating that the space of invariant symmetric $2$-tensors is one-dimensional.

    \item[(iii)] Because $G$ acts by isometries, the projection $\pi \colon (M^{\mathrm{princ}}, \ga) \to (M^{\mathrm{princ}}/G, \bar{\ga})$ is a Riemannian submersion onto the $1$-dimensional quotient manifold, with the orbits serving as fibers \cite{alexandrino2015lie}*{\S6.3}. Let $r$ denote an arc-length coordinate on $(M^{\mathrm{princ}}/G, \bar{\ga})$ and let $r$ also denote its pullback to $M^{\mathrm{princ}}$. The function $r$ is smooth, $G$-invariant, and has the orbits as level sets, with $|\nabla r|_{\ga} = 1$. For any vector field $X$ tangent to an orbit, $\ga(\nabla r, X) = \mathrm{d}r(X) = X(r) = 0$, showing that $\nabla r$ is normal to the orbits. The $\ga$-dual of $\nabla r$ is $\mathrm{d}r$, and since $\ga(\nabla r, \nabla r) = 1$, the metric splits orthogonally as $\ga = \mathrm{d}r^2 + \ga_r$, where $\ga_r := \ga|_{T(\text{orbit})}$. Because $(M,G)$ has an isotropy irreducible principal orbit, Item (ii) dictates that $\ga_r = \phi(r)^2 \ga_P$. The function $\phi$ is strictly positive and smooth because $\ga_r$ is a smooth positive-definite metric tensor.

    \item[(iv)] Consider the endpoint $r=0$, where the exceptional orbit is $Q = G/K$ with $K/H \cong \mathbb{Z}_2$. Because $H$ has index $2$ in $K$, it is normal, implying $K \subset N_G(H)$. Choose an element $w \in K$ representing the nontrivial coset, such that $w^2 \in H$ but $w \notin H$. A $G$-invariant tubular neighborhood of $Q$ is given by $G \times_K V$ with slice $V = \mathbb{R}$, on which $K$ acts via the homomorphism $K \to K/H \cong \mathbb{Z}_2 \hookrightarrow \mathrm{O}(1) = \{\pm1\}$. Because the orbit is exceptional, this map is surjective; thus, $w$ acts on $V$ by $v \mapsto -v$ and $H$ acts trivially. The associated two-fold cover
    \[
    \widehat{N} := G \times_H V = (G/H) \times (-\varepsilon, \varepsilon) \longrightarrow G \times_K V
    \]
    is an unbranched $G$-equivariant double cover (the two preimages of a point in $Q$ are distinct since $w \notin H$). The deck involution is given by
    \[
    \sigma(xH, s) = (\tau(xH), -s),\qquad \tau(xH) := xw^{-1}H,\qquad \tau^2 = \mathrm{id}\ \ (\text{since } w^2 \in H).
    \]
    Consider the pullback of $\ga$ to $\widehat{N}$. On this cover, the orbit over $Q$ is the principal orbit $(G/H) \times \{0\}$. By Item (iii), the pulled-back metric takes the form $\widehat{\ga} = \mathrm{d}s^2 + \widehat{\phi}(s)^2 \ga_P$ for $s \in (-\varepsilon, \varepsilon)$, where $\widehat{\phi}$ is smooth, positive, and satisfies $\widehat{\phi}(s) = \phi(s)$ for $s > 0$.

    Because $\widehat{\ga}$ descends to $M$, it must be $\sigma$-invariant. The map $\tau$ represents the action of $w \in N_G(H)$ on $G/H$, so $\tau^\ast \ga_P$ is another $G$-invariant metric. By Item (ii), $\tau^\ast \ga_P = c\,\ga_P$ for some constant $c > 0$. The condition $\tau^2 = \mathrm{id}$ implies $c^2 = 1$, thus $c = 1$. Consequently,
    \[
    \sigma^\ast\widehat{\ga} = \mathrm{d}s^2 + \widehat{\phi}(-s)^2\,\tau^\ast \ga_P = \mathrm{d}s^2 + \widehat{\phi}(-s)^2 \ga_P.
    \]
    The invariance $\sigma^\ast\widehat{\ga} = \widehat{\ga}$ necessitates $\widehat{\phi}(-s)^2 = \widehat{\phi}(s)^2$. Since $\widehat{\phi}$ is strictly positive, $\widehat{\phi}(-s) = \widehat{\phi}(s)$, demonstrating that $\widehat{\phi}$ is an even function. Being smooth and even, $\widehat{\phi}$ satisfies $\widehat{\phi}'(0) = 0$. This implies $\phi'(0) = 0$. The second fundamental form of the orbit is given by $II = \frac{1}{2}\partial_r \ga_r = \phi\phi'\,\ga_P$, which vanishes identically at $r=0$. Thus, $Q$ is totally geodesic. An identical argument at the endpoint $r=L$ yields $\phi'(L) = 0$.
\end{itemize}
\end{proof}

\begin{example}\label{ex:Bohm}
Fix $m \ge 2$ and consider the group diagram
$$H = SO(m)\ \subset\ K_- = K_+ = O(m)\ \subset\ G = SO(m+1),$$
where $O(m)\hookrightarrow SO(m+1)$ is defined by $A \mapsto \operatorname{diag}(\det A, A)$, and $H=SO(m)$ is its identity component. Let $M$ be the closed cohomogeneity-one $G$-manifold determined by this diagram \cites{Mostert, alexandrino2015lie}; it is the union of two copies of the twisted disk bundle $G \times_{O(m)} D^1$ glued along their common boundary $G/H \cong S^m$. Then $\dim M = m+1 \ge 3$, the principal orbit is $P = G/H = SO(m+1)/SO(m) \cong S^m$, and the two non-principal orbits are $G/K_\pm = SO(m+1)/O(m) \cong \mathbb{RP}^m$.

Because $K_\pm/H \cong O(m)/SO(m) \cong \mathbb{Z}_2 \cong S^0$, both non-principal orbits are \emph{exceptional}; they share the dimension $m$ with $P$, and $P$ acts as a two-fold covering over them. The action contains no zero-dimensional orbits, and the orbit space is $M/G \cong [0,L]$. The isotropy representation of $SO(m)$ on $\mathfrak{m} = T_{eH}S^m \cong \mathbb{R}^m$ is the standard representation, which is irreducible over $\mathbb{R}$ for $m \ge 2$. Thus, $P$ is isotropy irreducible and carries a unique $G$-invariant metric $\ga_P$ up to scaling, which is the round metric with $\operatorname{scal}_{\ga_P} = m(m-1) > 0$.

Following Example~\ref{eq:product-of}, isotropy irreducibility dictates that every $G$-invariant metric is a warped product $\ga = \mathrm{d}r^2 + \phi(r)^2 \ga_P$ on $[0,L]$ with $\phi > 0$. By Lemma~\ref{lem:smooth}(iv),
$$\phi(0),\,\phi(L) > 0, \qquad \phi'(0) = \phi'(L) = 0.$$
The scalar curvature is given by
$$\operatorname{scal}_{\ga} = \frac{m(m-1)}{\phi^2} - 2m\,\frac{\phi''}{\phi} - m(m-1)\,\frac{(\phi')^2}{\phi^2}, \qquad \mathrm{d}\mu_{\ga} = \phi^{m}\,\mathrm{d}r \wedge \mathrm{d}\mu_{\ga_P}.$$
Integration by parts on $[0,L]$ eliminates the boundary term $-2m\,[\phi'\phi^{m-1}]_0^L$ because $\phi'(0)=\phi'(L)=0$, yielding
$$\int_M \operatorname{scal}_{\ga}\,\mathrm{d}\mu_{\ga} = \operatorname{Vol}(P)\,m(m-1)\int_0^L \bigl[\phi^{m-2} + (\phi')^2\phi^{m-2}\bigr]\,\mathrm{d}r \; > \; 0.$$
Therefore, $(M, G)$ is totally $G$-positive, defined here with an interval as the orbit space.
\end{example}

\begin{lemma}\label{lem:dichotomy}
$(M, G)$ is totally $G$-positive if and only if it is of cohomogeneity one with an isotropy irreducible principal orbit. Moreover, if $(M,G)$ is not totally $G$-positive, then it admits a $G$-invariant metric $\ga_1$ satisfying $\int_M\mathrm{scal}_{\ga_1}\,\mathrm{d}\mu_{\ga_1}<0$.
\end{lemma}
\begin{proof}
Throughout the proof, $(M,G)$ is assumed to have no zero-dimensional orbits and $m := n-1 = \dim P \ge 2$.

\smallskip
\noindent\emph{$(\Leftarrow)$ Cohomogeneity one with irreducible isotropy $\Rightarrow$ totally $G$-positive.} 
Assume the action is of cohomogeneity one with an isotropy irreducible principal orbit $P=G/H$. By Lemma~\ref{lem:smooth}, the orbit space is $M/G \cong S^1$ or $[0,L]$, and every $G$-invariant metric is a warped product
$$\ga=\mathrm{d}r^2+\phi(r)^2\,\ga_P,\qquad \phi>0,$$
where $\ga_P$ is the unique-up-to-scale invariant metric on $P$. The function $\phi$ is either $L$-periodic (for $M/G\cong S^1$) or smooth and even across each endpoint with $\phi'(0)=\phi'(L)=0$ (for $M/G\cong[0,L]$).

By \cite{besse1987einstein}*{Cor.~7.44} and \cite{Wolf1968}, $(P,\ga_P)$ is Einstein, satisfying $\Ricci_{\ga_P}=\tfrac{\rho}{m}\,\ga_P$, where $\rho := \mathrm{scal}_{\ga_P}$ is constant. Furthermore, $\rho > 0$. If $\rho < 0$, Bochner's theorem \cite{Petersen}*{Ch.~7, Thm.~36 and Cor.~15} dictates that the isometry group of the compact manifold $(P,\ga_P)$ is finite, contradicting transitivity. If $\rho = 0$, the manifold $(P,\ga_P)$ is homogeneous and Ricci-flat, therefore flat according to Alekseevskii--Kimelfeld \cite{Alekseevskii1975}, \cite{besse1987einstein}*{Theorem~7.61}. A compact flat homogeneous space is a torus, which possesses a trivial isotropy representation. For $m \ge 2$, a trivial representation is reducible, yielding a contradiction.

The scalar curvature of $\ga=\mathrm{d}r^2+\phi^2\ga_P$ is given by \cite{besse1987einstein}*{Prop.~9.106} as
$$\mathrm{scal}_\ga=\frac{\rho}{\phi^2}-2m\frac{\phi''}{\phi}-m(m-1)\frac{(\phi')^2}{\phi^2}, \qquad \mathrm{d}\mu_\ga=\phi^{m}\,\mathrm{d}r \wedge \mathrm{d}\mu_{\ga_P},$$
which implies
$$\mathrm{scal}_\ga\,\mathrm{d}\mu_\ga=\bigl(\rho\,\phi^{m-2}-2m\,\phi''\phi^{m-1}-m(m-1)(\phi')^2\phi^{m-2}\bigr)\,\mathrm{d}r \wedge \mathrm{d}\mu_{\ga_P}.$$
Integrating the middle term by parts over the domain of length $L$ yields
$$\int_0^L\phi''\phi^{m-1}\,\mathrm{d}r=\bigl[\phi'\phi^{m-1}\bigr]_0^L-(m-1)\int_0^L(\phi')^2\phi^{m-2}\,\mathrm{d}r.$$
The boundary term vanishes in both scenarios: due to the $L$-periodicity of $\phi$ when $M/G\cong S^1$, and because $\phi'(0)=\phi'(L)=0$ when $M/G\cong[0,L]$ (Lemma~\ref{lem:smooth}(iv)). Hence, $-2m\int_0^L\phi''\phi^{m-1}\,\mathrm{d}r = 2m(m-1)\int_0^L(\phi')^2\phi^{m-2}\,\mathrm{d}r$. Combining this with the remaining term results in
$$\int_M\mathrm{scal}_\ga\,\mathrm{d}\mu_\ga=\mathrm{Vol}(P,\ga_P)\int_0^L\bigl(\rho\,\phi^{m-2}+m(m-1)(\phi')^2\phi^{m-2}\bigr)\,\mathrm{d}r>0,$$
since $\rho>0$ and $\phi>0$. Because $\ga$ was an arbitrary invariant metric, $(M,G)$ is totally $G$-positive.

\smallskip
\noindent\emph{(``Moreover'') If $(M,G)$ is not totally $G$-positive, it admits a $G$-invariant metric of negative total scalar curvature.} 
By the contrapositive of the sufficiency condition, such a pair is not of cohomogeneity one with irreducible isotropy. Therefore, either $\dim(M/G)\ge2$, or $\dim(M/G)=1$ with reducible principal isotropy.

If $\dim(M/G)\ge2$, Lemma~\ref{lem:negativetotalscal} provides a $G$-invariant $\ga_1$ satisfying $\int_M\mathrm{scal}_{\ga_1}\,\mathrm{d}\mu_{\ga_1}<0$.

If $\dim(M/G)=1$ with reducible isotropy, fix a nontrivial $\Ad(H)$-invariant splitting $\mathfrak{m}=\mathfrak{m}_1\oplus\mathfrak{m}_2$, where $d_i := \dim\mathfrak{m}_i \ge 1$ and $d_1+d_2=m$, along with an adapted invariant metric $\bar{g}=\bar{g}_1\oplus\bar{g}_2$ ($\bar{g}_i$ on $\mathfrak{m}_i$, $\mathfrak{m}_1\perp\mathfrak{m}_2$). The existence of such a metric is guaranteed by the compactness of $H$. Choose a principal orbit and a regular collar $U=P\times(a,b)$. One can construct an invariant reference metric $g_{\mathrm{ref}}$ that matches the product $\mathrm{d}r^2+\bar{g}$ on $U$ and closes smoothly at the non-principal orbits by averaging an arbitrary metric over $G$ and interpolating it with a cutoff function. For $t\in C^\infty_c(a,b)$, define
$$g_r := e^{2t/d_1}\,\bar{g}_1 \oplus e^{-2t/d_2}\,\bar{g}_2,\qquad \ga_1 := \begin{cases}\mathrm{d}r^2+g_r,&\text{on }U,\\ g_{\mathrm{ref}},&\text{on }M\setminus U.\end{cases}$$
Because $t$ has compact support, $\ga_1=\mathrm{d}r^2+\bar{g}$ near $\partial U$ and $\ga_1=g_{\mathrm{ref}}$ outside the support region. Thus, $\ga_1$ is a globally smooth $G$-invariant metric that remains fixed at $g_{\mathrm{ref}}$ near every non-principal orbit. 

Because $\det g_r=\det\bar{g}$, the volume is preserved, $\mathrm{Vol}(P,g_r)\equiv V := \mathrm{Vol}(P,\bar{g})$, and the orbits are minimal, $H=\partial_r\ln\sqrt{\det g_r}=0$. The shape operator $S=\frac{1}{2} g_r^{-1}\partial_r g_r$ restricts to $\frac{t'}{d_1}\id$ on $\mathfrak{m}_1$ and $-\frac{t'}{d_2}\id$ on $\mathfrak{m}_2$. Consequently, on $U$,
$$H=0,\qquad |II|^2=\mathrm{tr}(S^2)=(t')^2\Bigl(\frac{1}{d_1}+\frac{1}{d_2}\Bigr).$$
With $\partial_r$ acting as a unit geodesic normal field, the contracted Gauss equation and the trace of the radial Riccati equation \cite{Petersen}*{Ch.~2, \S4} imply
$$\mathrm{scal}_{\ga_1}=\mathrm{scal}_{g_r}-2\,\partial_r H-H^2-|II|^2=\mathrm{scal}_{g_r}-(t')^2\Bigl(\frac{1}{d_1}+\frac{1}{d_2}\Bigr)\quad\text{on }U,$$
where $\mathrm{scal}_{g_r}$ depends on $r$ exclusively through $t(r)$. Because $\mathrm{Vol}(P,g_r)\equiv V$ and $\ga_1=g_{\mathrm{ref}}$ outside $U$,
$$\int_M\mathrm{scal}_{\ga_1}\,\mathrm{d}\mu_{\ga_1}=A+V\!\int_a^b\!\mathrm{scal}_{g_r}\,\mathrm{d}r-V\Bigl(\frac{1}{d_1}+\frac{1}{d_2}\Bigr)\!\int_a^b\!(t')^2\,\mathrm{d}r,$$
where $A := \int_{M\setminus U}\mathrm{scal}_{g_{\mathrm{ref}}}\,\mathrm{d}\mu_{g_{\mathrm{ref}}}$.

Restrict the deformation to $\|t\|_{C^0}\le 1$. The metric $g_r$ then ranges over a compact family, implying a uniform bound $|\mathrm{scal}_{g_r}|\le C_1$. Consequently, the first two integral terms are bounded by $A+VC_1(b-a)$ independently of $t'$. Defining $t(r) := \chi(r)\sin(\omega r)$ with a test function $\chi\in C^\infty_c(a,b)$ that equals $1$ on a central subinterval guarantees $\|t\|_{C^0}\le 1$, whereas $\int_a^b(t')^2\,\mathrm{d}r\to\infty$ as $\omega\to\infty$. Therefore, $\int_M\mathrm{scal}_{\ga_1}\,\mathrm{d}\mu_{\ga_1}\to-\infty$, ensuring strict negativity for sufficiently large $\omega$.

Consequently, the sufficiency condition and its contrapositive establish both the equivalence and the existence of the strictly negative scalar curvature metric.
\end{proof}

\ 

\begin{proof}[Proof of Theorem~\ref{theorem:invariantyamabe}]
\textbf{(2)} This is the content of Proposition~\ref{thm:invariantkazdan}.

\textbf{(3)} If the action is of cohomogeneity one with isotropy-irreducible
principal orbit, Lemma~\ref{lem:dichotomy} shows $(M,G)$ is totally $G$-positive, i.e.
$\int_M\mathrm{scal}_\ga\,\mathrm{d}\mu_\ga>0$ for every $G$-invariant $\ga$. Any $G$-invariant metric $\ga$ with $\mathrm{scal}_\ga\le 0$ everywhere would satisfy
$\int_M\mathrm{scal}_\ga\,\mathrm{d}\mu_\ga\le 0$, contradicting total $G$-positivity. Hence, $M$ carries no $G$-invariant metric with everywhere non-positive scalar curvature.

Moreover, the product metric $\mathrm{d} r^2+\ga_P$ is a smooth $G$-invariant metric on $M$: in the
circle case this is immediate, while in the interval case its pullback $\mathrm{d} s^2+\ga_P$ to the
double cover $\widehat N$ of Lemma~\ref{lem:smooth}(iv) is invariant under the deck involution
$\sigma$ (since $\tau^*\ga_P=\ga_P$), hence descends smoothly across each exceptional orbit.
Its scalar curvature is the constant $\rho>0$. Hence, $(M, G)$ admits a metric
of positive constant scalar curvature.

\textbf{(1)} Assume $(M,G)$ is not totally $G$-positive.
By Lemma~\ref{lem:dichotomy} there is a $G$-invariant metric $\ga_1$ with
$\int_M\mathrm{scal}_{\ga_1}\,\mathrm{d}\mu_{\ga_1}<0$, and Theorem~\ref{thm:negativeexistence} produces
a $G$-invariant metric $\ga_-$ with $\mathrm{scal}_{\ga_-}\equiv-1$. This proves the first claim.

Suppose moreover $(M,G)$ admits a $G$-invariant metric $\ga$ with $\mathrm{scal}_\ga\ge0$. If
$\mathrm{scal}_\ga\equiv0$ we are done. Otherwise, $\mathrm{scal}_\ga$ is non-negative and not
identically zero, so by (2) there is a $G$-invariant metric $\ga_+$ with
$\mathrm{scal}_{\ga_+}\equiv k>0$. Consider the segment of $G$-invariant metrics
\[
\ga_t=(1-t)\ga_-+t\,\ga_+,\qquad t\in[0,1],
\]
each smooth and $G$-invariant since the invariant metrics form an open convex cone. The lowest
eigenvalue of the conformal Laplacian $L_t=-4b_n\Delta_{\ga_t}+\mathrm{scal}_{\ga_t}$,
\[
\lambda_1(t):=\inf_{0\not\equiv u\in W^{1,2}(M)}
\frac{\int_M\!\big(4b_n|\nabla u|_{\ga_t}^2+\mathrm{scal}_{\ga_t}u^2\big)\,\mathrm{d}\mu_{\ga_t}}
{\int_M u^2\,\mathrm{d}\mu_{\ga_t}},
\]
depends continuously on $t$. Testing $u\equiv1$ at $t=0$ gives
$\lambda_1(0)\le\int_M(-1)\,\mathrm{d}\mu_{\ga_0}/\int_M\mathrm{d}\mu_{\ga_0}=-1<0$, while at $t=1$ the gradient term is non-negative and $\mathrm{scal}_{\ga_+}\equiv k$,
so $\lambda_1(1)\ge k>0$. By
the intermediate value theorem $\lambda_1(t_0)=0$ for some $t_0\in(0,1)$.

The first eigenvalue of $L_{t_0}$ being $0$, the maximum principle gives a positive smooth
eigenfunction $u$ spanning the (one-dimensional) eigenspace. Since $\ga_{t_0}$ is
$G$-invariant, $L_{t_0}$ commutes with the $G$-action, so $g^*u=c(g)u$ with $c(g)>0$. Because
$c\colon G\to(\mathbb R_{>0},\times)$ is a continuous homomorphism with compact image, we have that
$c\equiv1$ and thus $u$ is $G$-invariant. The metric $\widetilde\ga=u^{4/(n-2)}\ga_{t_0}$ is
$G$-invariant and $\mathrm{scal}_{\widetilde\ga}=u^{-\frac{n+2}{n-2}}L_{t_0}u=0.$\end{proof}

\ 

\section{Proof of Theorem \ref{ithm:main}}
\label{sec:classification}

\textbf{(1) Totally $G$-positive.} By Lemma~\ref{lem:dichotomy} and Theorem~\ref{theorem:invariantyamabe}(3), $(M, G)$ admits a $G$-invariant metric of positive constant scalar curvature, but no $G$-invariant metric with $\mathrm{scal}\le0$ everywhere. Equivalently,
(total $G$-positivity) every invariant scalar curvature is positive somewhere. Thus no $f\le0$ is
realized, so $(M,G)$ lies in none of $\mathscr P^G,\mathscr Z^G,\mathscr N^G$.

\medskip
\textbf{(2) Not totally $G$-positive.} Theorem~\ref{theorem:invariantyamabe}(1) ensures the existence of $\ga_{-1}$
with $\mathrm{scal}_{\ga_{-1}}\equiv-1$. The classification rests on one construction.

\begin{lemma}\label{lem:realization}
    Fix $C\in\{+1,-1\}$ for which a $G$-invariant metric
$\ga_C$ with $\mathrm{scal}_{\ga_C}\equiv C$ exists. Then every non-constant $f\in C^\infty_G(M)$ with
$C\bar f>0$ somewhere on $X^*$ is the scalar curvature of a $G$-invariant metric.
\end{lemma}
\begin{proof}
If $\dim X^*\ge2$, choose $c>0$ with $c\min_M f<C<c\max_M f$. Notice that such a $c$ exists:
if $C=+1$, then $\max_M f>0$, and any $c>1/\max_M f$ works when $\min_M f\le 0$, while
any $c\in(1/\max_M f,\,1/\min_M f)$ works when $\min_M f>0$. The case $C=-1$ is symmetric. Since
$\mathrm{scal}_{\ga_C}\equiv C\in(c\min_M f,\,c\max_M f)$, Theorem~\ref{theorem:GKW}(a) applied
to $\ga_C$ realizes $f$.

If $\dim X^*=1$ (where $X\cong[0, L]$ or $S^1$), choose $x_0\in X^*$ such that $C\bar f(x_0)>0$,
and define $c:=C/\bar f(x_0)>0$. Fix $\eta>0$ and a smooth weakly monotone surjective map
$h\colon X\to X$. In the interval case, require $h$ to be locally constant at $0$ and $L$ on a
neighborhood of the boundary, and equal to the constant $x_0$ outside a set whose measure
is bounded by a function that vanishes as $\eta \to 0$. In the circle case, require $h$ to have degree one and equal $x_0$ outside a set whose measure satisfies the identical requirement.

Since $h$ is locally constant near the non-principal orbits, the function
$c\,(\bar f\circ h)\circ\pi$ is smooth on $M$. Then
$\bar f\circ h$ is non-constant and $c\,(\bar f\circ h)\circ\pi\to C$ in $L^p$ as $\eta\to0$. By
Lemma~\ref{lem:dense_kernel} perturb $\ga_C$ to $\ga$ with $\ker A^*_\ga=\{0\}$ and $\mathrm{scal}_\ga$
arbitrarily $L^p$-close to $C$. Lemma~\ref{lem:inversefunction} yields a $G$-invariant
$\ga^*$ with $\mathrm{scal}_{\ga^*}=c\,(\bar f\circ h)\circ\pi$ once $\eta$ and the perturbation are small,
i.e.\ $\overline{\mathrm{scal}}_{\ga^*}=c\,\bar f\circ h$. Theorem~\ref{theorem:GKW}(b) (resp.\ (c))
realizes $f$. 
\end{proof}

\smallskip
\noindent Since $\ga_{-1}$ always provides $C=-1$, Lemma \ref{lem:realization} realizes \emph{every
non-constant $f$ that is negative somewhere}. If a positive-constant-scalar metric also exists,
$C=+1$ realizes \emph{every non-constant $f$ positive somewhere}. Constant $f$ is realized by (possibly after scaling)
whichever of the positive / zero / negative constant-scalar metrics exists.

Next, recall that each pair admits a metric $\ga_{-1}$ with constant negative scalar curvature. Consequently, for any pair $(M,G)$, exactly one of the following conditions holds:
\textbf{Category \,1.} a positive-constant-scalar metric exists; \textbf{Category \,2.} a metric with
$\mathrm{scal}\equiv0$ exists, but no positive constant one does; \textbf{Category \,3.} no metric with $\mathrm{scal}\ge0$ exists.
By Proposition~\ref{thm:invariantkazdan}, any metric with $\mathrm{scal}\ge0$ but $\not\equiv0$ forces
Cat.\,1. Thus, in Cats.\,2--3 no such metric exists.

\emph{Cat.\,1 $\Rightarrow\mathscr P^G$.} Both $C=\pm1$ are available, and every non-constant $f$ is
positive or negative somewhere, hence realized. Every constant is realized (positive by Cat.\,1,
zero by Theorem~\ref{theorem:invariantyamabe}(1), negative by $\ga_{-1}$).

\emph{Cat.\,2 $\Rightarrow\mathscr Z^G$.} $f\equiv0$ uses the zero-scalar metric; $f$ negative
somewhere uses the Lemma \ref{lem:realization} with $C=-1$. Conversely, a realized $f\ge0$ which is not identically zero 
would force Cat.\,1 (Proposition~\ref{thm:invariantkazdan}). So $f$ is realized if, and only if, $f\equiv0$ or
$f$ is negative somewhere.

\emph{Cat.\,3 $\Rightarrow\mathscr N^G$.} $f$ negative somewhere uses the Lemma \ref{lem:realization} with
$C=-1$. Conversely, no metric has $\mathrm{scal}\ge0$, so every metric, and hence every realized $f$, is
negative somewhere.

\medskip
\textbf{Non-abelian case.} If $\mathfrak g$ is non-abelian, then by Lawson--Yau
\cite{lawson-yau}*{Theorem~2} the effective action of the compact connected non-abelian group
$G$ admits a $G$-invariant metric of positive scalar curvature. Hence, by
Theorem~\ref{theorem:invariantyamabe}(2), $(M,G)$ carries a $G$-invariant metric of positive
constant scalar curvature, so $(M,G)$ is in Cat.\,1 and $(M,G)\in\mathscr P^G$. \hfill$\square$

\bigskip

\begin{flushleft}
 {\bf Funding:} Part of this work was developed between the end of the first-named author's (L. F. Cavenaghi) postdoctoral position at the Federal University of Paraíba (UFPB) and the University of Fribourg, supported partly by the SNSF-Project 200020E\_193062 and the DFG-Priority Program SPP 2026. He takes the opportunity to thank the University of Fribourg for its hospitality, as well as the freedom and confidence of Prof. Anand Dessai during the period he supervised him. L. F. Cavenaghi was supported by The São Paulo Research Foundation (FAPESP), grants 2023/14316-1 and 2022/09603-9, and currently is supported by the Simons Foundation, grant SFI-MPS-T-Institutes-00007697, and the Ministry of Education and Science of the Republic of Bulgaria, grant DO1-239/10.12.2024.  J. M. do \'O acknowledges partial support from CNPq through grants 312340/2021-4, 409764/2023-0, 443594/2023-6, CAPES MATH AMSUD 
 grant 88887.878894/2023-00
and Para\'iba State Research Foundation (FAPESQ), grant no 3034/2021.     \\ 
 {\bf Ethical Approval:}  Not applicable.\\
 {\bf Competing interests:}  Not applicable. \\
 \noindent\textbf{Authors' contributions.}
All authors contributed to the conception of the work, the development of the
mathematical arguments, and the writing and revision of the manuscript. All
authors read and approved the final manuscript.\\
{\bf Availability of data and material:}  Not applicable.\\
{\bf Consent to participate:}  All authors consent to participate in this work.\\
{\bf Conflict of interest:} The authors declare no conflict of interest. \\
{\bf Consent for publication:}  All authors consent for publication. \\
\end{flushleft}

\bibliographystyle{plain}
\bibliography{references}
\end{document}